\documentclass[11pt]{article}
\usepackage{amsfonts,amsmath,amssymb,verbatim}
\usepackage{hyperref,amsthm}
\usepackage{epsfig}
\usepackage{graphicx}
\newtheorem{theorem}{Theorem}[section]
\newtheorem{proposition}[theorem]{Proposition}
\newtheorem{lemma}[theorem]{Lemma}
\newtheorem{claim}[theorem]{Claim}
\newtheorem{corollary}[theorem]{Corollary}

\newtheorem{conjecture}[theorem]{Conjecture}

\newtheorem*{theorem*}{Theorem}
\newtheorem*{lemma*}{Lemma}
\newtheorem*{proposition*}{Proposition}
\newtheorem*{corollary*}{Corollary 1.17}
\newtheorem*{theorem-frankl}{Theorem 1.6}

\theoremstyle{definition}
\newtheorem{remark}[theorem]{Remark}
\newtheorem{notation}[theorem]{Notation}
\newtheorem{definition}[theorem]{Definition}

\DeclareMathOperator{\Inf}{Inf}
\DeclareMathOperator{\Bin}{Bin}

\newcommand{\beq}[1]{\begin{equation}\label{#1}}
\newcommand{\enq}[0]{\end{equation}}

\newcommand{\mn}[0]{\medskip\noindent}

\newcommand{\A}[0]{{\cal A}}
\newcommand{\B}[0]{{\cal B}}
\newcommand{\F}[0]{{\cal F}}
\newcommand{\C}[0]{{\cal C}}
\newcommand{\D}[0]{{\cal D}}
\newcommand{\G}[0]{{\cal G}}
\newcommand{\h}[0]{{\cal H}}

\newcommand{\p}[0]{{\cal P}}
\newcommand{\s}[0]{{\cal S}}

\newcommand{\OR}{\mathrm{OR}}

\usepackage[usenames]{color}

\begin{document}

\title{Stability versions of Erd\H{o}s-Ko-Rado type theorems, via isoperimetry}

\author{
David Ellis\thanks{School of Mathematical Sciences, Queen Mary, University of London, Mile End Road, London, E1 4NS, UK. {\tt d.ellis@qmul.ac.uk}},
Nathan Keller\thanks{Department of Mathematics, Bar Ilan University, Ramat Gan, Israel.
{\tt nathan.keller27@gmail.com}. Research supported by the Israel Science Foundation (grant no.
402/13), the Binational US-Israel Science Foundation (grant no. 2014290), and by the Alon Fellowship.},
and Noam Lifshitz\thanks{Department of Mathematics, Bar Ilan University, Ramat Gan, Israel.
{\tt noamlifshitz@gmail.com}.}
}

\maketitle

\begin{abstract}
Erd\H{o}s-Ko-Rado (EKR) type theorems yield upper bounds on the sizes of
families of sets, subject to various intersection requirements on the
sets in the family. Stability versions of such theorems assert that if the size of a family
is close to the maximum possible size, then the family itself must be close (in some appropriate sense)
to a maximum-sized family.

In this paper, we present an approach to obtaining stability versions of EKR-type
theorems, via isoperimetric inequalities for subsets of the hypercube. Our approach is rather general,
and allows the leveraging of a wide variety of exact EKR-type results into strong stability versions of these results, without
going into the proofs of the original results.

We use this approach to obtain tight stability versions of the EKR
theorem itself and of the Ahlswede-Khachatrian theorem on $t$-intersecting families of $k$-element subsets of $\{1,2,\ldots,n\}$ (for $k < \frac{n}{t+1}$), and to show that, somewhat surprisingly, all these results hold when the
`intersection' requirement is replaced by a much weaker requirement.

Other examples include stability versions of Frankl's recent result on the Erd\H{o}s matching conjecture,
the Ellis-Filmus-Friedgut proof of the Simonovits-S\'{o}s conjecture, and various EKR-type results on
$r$-wise (cross)-$t$-intersecting families.
\end{abstract}

\newpage

\tableofcontents

\newpage

\section{Introduction}

\subsection{Background}

Let $[n]:=\{1,2,\ldots,n\}$, and let $[n]^{(k)}:=\{A \subset [n]:\ |A|=k\}$. A family $\F \subset \p([n])$
(i.e., a family of subsets of $[n]$) is said to be \emph{intersecting} if for
any $A,B \in \F$, $A \cap B \neq \emptyset$. For $t \in \mathbb{N}$, a family $\F \subset \p([n])$ is said to be
\emph{t-intersecting} if for any $A,B \in \F$, $|A \cap B| \geq t$.

A {\em dictatorship} is a family of the form $\{S \subset [n]:\ j \in S\}$ or $\{S \subset [n]^{(k)}:\ j \in S\}$ for some $j \in [n]$.
A {\em $t$-umvirate} is a family of the form $\{S \subset [n]:\ B \subset S\}$ or $\{S \subset [n]^{(k)}:\ B \subset S\}$ for some $B \in [n]^{(t)}$. For $B \subset [n]$, we write $\s_{B} := \{S \subset [n]:\ B \subset S\}$ for the corresponding $|B|$-umvirate.

One of the best-known theorems in extremal combinatorics is the Erd\H{o}s-Ko-Rado (EKR)
theorem~\cite{EKR}:
\begin{theorem}[Erd\H{o}s-Ko-Rado, 1961]
\label{thm:ekr}
Let $k,n \in \mathbb{N}$ with $k < n/2$, and let $\F \subset [n]^{(k)}$ be an intersecting family. Then $|\F| \leq {{n-1}\choose{k-1}}$.
Equality holds if and only if $\F= \{S \in [n]^{(k)}: j \in S\}$ for some $j \in [n]$.
\end{theorem}

The Erd\H{o}s-Ko-Rado theorem is the starting point of an entire subfield of extremal combinatorics, concerned with bounding the sizes of families of sets, under various intersection requirements on sets in the family. Such results are often called \emph{EKR-type results}.

For more background and history on EKR-type results, we refer the reader to the surveys~\cite{DF83,MV15} and the references therein. For our purposes here, we only mention the probably best-known EKR-type result --- the Ahlswede-Khachatrian (AK) theorem~\cite{AK}, which bounds the size of $t$-intersecting families of $k$-element sets.
\begin{theorem}[Ahlswede-Khachatrian, 1997]
\label{thm:ak}
For any $n,k,t,r \in \mathbb{N}$, let $\F_{n,k,t,r} := \{S \in [n]^{(k)}: |S \cap [t+2r]| \geq t+r\}$.
Let $\F \subset [n]^{(k)}$ be a $t$-intersecting family. Then $|\F| \leq \max_r |\F_{n,k,t,r}|$.
\end{theorem}
The $n \geq (t+1)(k-t+1)$ case of Theorem \ref{thm:ak}, in which $\max_r |\F_{n,k,t,r}| = |\F_{n,k,t,0}| = {n-t \choose k-t}$, was proved earlier, by Wilson~\cite{Wilson}:
\begin{theorem}[Wilson, 1984]
\label{Thm:Wilson}
Let $n,k,t \in \mathbb{N}$ such that $n \geq (t+1)(k-t+1)$. Let $\F \subset [n]^{(k)}$ be a $t$-intersecting family. Then $|\F| \leq {n-t \choose k-t}$. If $n > (t+1)(k-t+1)$, then equality holds if and only if  $\mathcal{F} = \{S \in [n]^{(k)}:\ B \subset S\}$ for some $B \in [n]^{(t)}$.
\end{theorem}

Over the years, numerous authors have obtained \emph{stability versions} of EKR-type results, asserting that if the size of a family is `close' to the maximum possible size, then that family is `close' (in some appropriate sense) to an extremal family.

Perhaps the first such stability result was obtained in 1967 by Hilton and Milner \cite{HM67}; they showed that if the size of an intersecting family is \emph{very} close to ${{n-1}\choose{k-1}}$, then the family is {\it contained} in a dictatorship. A corresponding result for Wilson's theorem was obtained in 1996 by Ahlswede and Khachatrian~\cite{AK96}. A simpler proof of the latter result was presented by Balogh and Mubayi~\cite{BM08}, and an alternative result of the same class was obtained by Anstee and Keevash~\cite{AK06}.

For families whose size is not very close to the maximum, Frankl~\cite{Frankl87} obtained in 1987 a rather strong stability version of the EKR theorem, proving that if an intersecting family $\mathcal{F}$ satisfies $|\mathcal{F}| \geq (1-\epsilon){{n-1}\choose{k-1}}$, then there exists a dictatorship $\mathcal{D}$ such that $|\F \setminus \mathcal{D}| = O(\epsilon^{\log_{1-p}p}) {{n}\choose{k}}$, where $p \approx k/n$. Frankl's result, obtained using {\em combinatorial shifting} (a.k.a. {\em compression}), is best-possible, and holds not only for $|\mathcal{F}|$ close to ${{n-1}\choose{k-1}}$ but rather for any $|\mathcal{F}| \geq 3{n-2\choose k-2} - 2{n-3 \choose k-3}$. Proofs of somewhat weaker results using entirely different techniques were later presented by Dinur and Friedgut~\cite{DF09}, Friedgut~\cite{Friedgut08}, and Keevash~\cite{Keevash08}. In~\cite{KM10}, Keevash and Mubayi used Frankl's result to prove an EKR-type theorem on set systems that do not contain a simplex or a cluster. Recently, a different notion of
stability for the EKR theorem was suggested by Bollob\'as, Narayanan and Raigorodskii in \cite{BNR16}; this has already been studied in several subsequent papers (e.g.,~\cite{BBN16+,DK16+}).

In 2008, Friedgut~\cite{Friedgut08} used spectral methods and Fourier analysis to obtain the following stability version of Wilson's theorem:
\begin{theorem}[Friedgut, 2008]
\label{thm:Fr-uniform}
For any $t \in \mathbb{N}$ and any $\eta >0$, there exists $c=c(t,\eta)>0$ such that the following holds. Let $\eta n < k < (1/(t+1)-\eta)n$ and let $\epsilon \geq \sqrt{(\log n) / n}$. If $\F \subset [n]^{(k)}$ is a $t$-intersecting family with $|\F| \geq (1-\epsilon){n - t \choose k-t}$, then there exists $B \in [n]^{(t)}$ such that
$|\F \setminus \s_{B}| \leq c \epsilon {n \choose k}$.
\end{theorem}

Many other stability versions of EKR-type results have been obtained in recent years (see e.g. \cite{FLST14,Kamat11,KMW07,Mubayi07,Tokushige10,Tokushige11}). Besides being interesting in their own right, such stability results often serve as a route for proving exact EKR-type results. (In the more general setting of Tur\'{a}n-type problems, the idea of using a stability result to obtain an exact result goes back perhaps to Simonovits~\cite{Simonovits68}.)

\subsection{Our results}

In this paper, we present a new method for obtaining stability versions of EKR-type results. Our method, based on isoperimetric inequalities on the hypercube (see below), is rather general, and allows the leveraging of a wide range of exact EKR-type results into stability results, without going into their proofs. It works whenever the following two conditions are satisfied:
\begin{itemize}
\item We have a `starting point' --- an exact EKR-type result for families $\F \subset [n]^{(k_0)}$ for some $k_0 \leq n$ (here, $k_0$ may depend on $n$).

\item The extremal example is either a $t$-umvirate, or the dual of a $t$-umvirate (that is, a family of the form $\{S \subset [n]^{(k_0)}: S \cap B \neq \emptyset\}$, for some $B \in [n]^{(t)}$).
\end{itemize}
Given these two conditions, our method allows us to deduce a strong stability result that holds for families $\F \subset [n]^{(k)}$ whenever $k \leq (1-\eta)k_0$, for any fixed $\eta>0$.

\medskip \noindent The two conditions above hold in a wide variety of settings, including:
\begin{enumerate}
\item The EKR theorem itself (where the extremal example is a dictatorship),

\item Wilson's theorem on $t$-intersecting families (where the extremal example is a $t$-umvirate),

\item The Simonovits-S\'{o}s conjecture on triangle-intersecting families of graphs, proved recently by Ellis, Filmus and Friedgut \cite{EFF12} (where the extremal example is a specific type of $3$-umvirate, namely all graphs containing a fixed triangle),

\item The Erd\H{o}s matching conjecture on the maximal size of a family without $s$ pairwise disjoint sets, proved recently by Frankl~\cite{Frankl13} for $n>(2s+1)k-s$ (where the extremal example is the dual of an $(s-1)$-umvirate),

\item All EKR-type results on $r$-wise (cross)-$t$-intersecting families known to us, for which the extremal example is a $t$-umvirate. (See e.g.\ \cite{FLST14,Kamat11,Tokushige10,Tokushige11} and the references therein.)
\end{enumerate}

In all these cases, our method leads to stability results which are much stronger than those obtained in several previous works (e.g.,~\cite{DF09,EFF12,Friedgut08,Kamat11,Tokushige11}).

\medskip For example, we obtain the following stability version of Wilson's theorem:
\begin{theorem}\label{thm:stability-uniform-t}
For any $t \in \mathbb{N}$ and $\eta >0$, there exists $\delta_{0}=\delta_0(\eta,t)>0$ such that the following holds. Let $n,k \in \mathbb{N}$ with $k \leq (\tfrac{1}{t+1}-\eta)n$, and let $d\in\mathbb{N}$. Let $\mathcal{A}\subset [n]^{(k)}$ be a $t$-intersecting family with
\begin{equation} \label{eq:wilson-stability-1}
\left|\A\right| > \max\left\{ \binom{n-t}{k-t}\left(1-\delta_{0}\right),\binom{n-t}{k-t}-
\binom{n-t-d}{k-t}+\left(2^{t}-1\right)\binom{n-t-d}{k-t-d+1}\right\}.
\end{equation}
Then there exists a $t$-umvirate $\s_B$ such that
\begin{equation}\label{eq:wilson-stability-2}
\left|\mathcal{A} \setminus \mathcal{S}_{B} \right| \leq \left(2^{t}-1\right)\binom{n-t-d}{k-t-d+1}.
\end{equation}
\end{theorem}
\noindent

Theorem~\ref{thm:stability-uniform-t} improves significantly over Theorem~\ref{thm:Fr-uniform}. In particular, it implies the following in the case where $k = \Theta(n)$.

\begin{corollary}
\label{corr:rough}
Let $n,k,t \in \mathbb{N}$ with $\eta n \leq k \leq (\tfrac{1}{t+1}-\eta)n$, let $\epsilon >0$, and let $\mathcal{A}\subset [n]^{(k)}$ be a $t$-intersecting family with
$$\left|\A\right| \geq (1-\epsilon)\binom{n-t}{k-t}.$$
Then there exists a $t$-umvirate $\s_B$ such that
$$\left|\mathcal{A} \setminus \mathcal{S}_{B} \right| = O_{t,\eta}(\epsilon^{\log_{1-k/n}(k/n)}) \binom{n-t}{k-t}.$$
\end{corollary}

The $\epsilon$-dependence in Corollary \ref{corr:rough} is tight up to a constant factor depending upon $t$ and $\eta$ alone. Moreover, for $d$ sufficiently large (as function of $t$ and $\eta$), Theorem~\ref{thm:stability-uniform-t} is tight (even for $k=o(n)$), up to
replacing $2^t-1$ with $t$ in the inequalities (\ref{eq:wilson-stability-1}) and (\ref{eq:wilson-stability-2}), as evidenced by the families $(\mathcal{F}_{t,s})_{t,s \in \mathbb{N}}$, defined by
\begin{align*}
\mathcal{F}_{t,s} & :=\left\{ A\subset [n]^{(k)}\,:\,\left[t\right]\subset A,\,\left\{t+1,\ldots,t+s\right\}\cap A\ne\emptyset\right\} \\
 & \cup\left\{ A\subset [n]^{(k)}\,:\,|\left[t\right]\cap A|=t-1,\,\left\{t+1,\ldots,t+s\right\}\subset A\right\}.
\end{align*}

\subsection{Our methods}
\label{sec:sub:methods}
Following the works of Dinur and Friedgut~\cite{DF09} and Friedgut~\cite{Friedgut08}, we first obtain stability versions of EKR-type results in the {\em biased-measure} setting, and we then leverage them to the `classical' setting of $k$-element sets. Let us briefly recall some known EKR-type upper bounds in the biased-measure setting.

\subsubsection*{EKR-type results in the $p$-biased setting}
For $p\in [0,1]$, the {\em $p$-biased measure on $\p([n])$} is defined by
$$\mu_p(S) = p^{|S|} (1-p)^{n-|S|}\quad \forall S \subset [n].$$
In other words, we choose a random set by including each $j \in [n]$ independently with probability $p$. For $\F \subset \p([n])$, we define $\mu_p(\F) = \sum_{S \in \F} \mu_p(S)$.

The $p$-biased version of the EKR theorem is as follows.
\begin{theorem}[Biased EKR Theorem]\label{Thm:Biased-EKR}
Let $0<p\leq1/2$. Let $\F \subset \mathcal{P}([n])$ be an intersecting family. Then $\mu_p(\F) \leq p$. If $p < 1/2$, then equality holds if and only if $\F = \{S \subset [n]:\ j \in S\}$ for some $j \in [n]$.
\end{theorem}
First obtained by Ahlswede and Katona~\cite{AK77} in 1977, Theorem~\ref{Thm:Biased-EKR} was reproved numerous times using various different techniques (see e.g.\ \cite{DF06,F07,FFFLO86,FT03}).

Over the years, $p$-biased versions of several different EKR-type upper bounds have been obtained. In fact, it is known that they can be deduced from the corresponding EKR-type results for $k$-element sets, with $k \approx pn$, using a method known as `going to infinity and back' (first used by Dinur and Safra \cite{Dinur-Safra} and independently by Frankl and Tokushige \cite{FT03}). For example, this method can be used to deduce from Theorem \ref{Thm:Wilson} the following biased version of Wilson's theorem.
\begin{theorem}
\label{thm:Wilson-biased}
Let $\mathcal{F} \subset \p([n])$ be a $t$-intersecting family, and let $0 \leq p \leq 1/(t+1)$. Then $\mu_p(\mathcal{F}) \leq p^{t}$. If $p < 1/(t+1)$, equality holds if and only if $\mathcal{F}$ is a $t$-umvirate.
\end{theorem}

In the other direction, a sharp upper bound for an EKR-type problem in the biased-measure setting implies only an approximate version of the corresponding `classical' ($k$-uniform) result. Similarly, in this paper, we obtain rather sharp stability results in the $p$-biased setting, but on their own, these results imply only `rough' stability results in the $k$-uniform setting. Fortunately, in the cases we consider, we are able to combine these rough stability results with some additional combinatorial arguments to obtain much sharper stability results in the $k$-uniform setting.

\subsubsection*{Our stability results in the $p$-biased setting}
One of our key ideas is as follows. Instead of working only with intersecting families, we consider the larger class of {\em increasing} families. (A family $\mathcal{F} \subset \p([n])$ is said to be {\em increasing} if it is closed under taking supersets, i.e.\ whenever $A \subset B$ and $A \in \mathcal{F}$, we have $B \in \mathcal{F}$. It is clear that in proving any $p$-biased EKR-type theorem giving an upper bound on the $p$-biased measure of a family $\F$, we can assume without loss of generality that $\F$ is increasing.) It turns out that our stability results in the biased-measure setting hold for increasing families as well as for intersecting families --- i.e., under a much weaker requirement. Hence, in the biased-measure part of our paper, all of our stability results are stated and proved for increasing families; this allows us to give unified proofs of stability versions of several different EKR-type results.

As our main biased stability theorem (Theorem \ref{Thm:Main}) is somewhat technical, we delay its statement until Section~\ref{sec:proof}. Here, we present a special case (the dictatorship case), whose statement is simpler.

\begin{theorem}\label{Thm:Biased1-intro}
There exist absolute constants $c>0$ and $C>4$ such that the following holds. Let $0 < p<1/2$, and let $\mathcal{F} \subset \p([n])$ be an increasing family such that $\mu_{1/2}(\F) \leq 1/2$ and
\begin{equation}\label{eq:hypo-intro}
\mu_p(\F) \geq \begin{cases} p(1-c(1/2-p)) & \text{ if } p \geq 1/C\\ Cp^{2} & \text{ if } p < 1/C.\end{cases}
\end{equation}
Let $\epsilon>0$. If
\begin{equation}
\label{Eq:Cor-Condition-intro}
\mu_{p}\left(\mathcal{F}\right)\ge p\left(1-\epsilon^{\log_{p}(1-p)}\right)+\left(1-p\right)\epsilon,
\end{equation}
then $\mu_p(\F \setminus \F_j) \leq (1-p)\epsilon$ for some dictatorship $\F_j$.
\end{theorem}

Theorem~\ref{Thm:Biased1-intro} is tight for the families $\{\tilde{\G}_i\}_{i \geq 3}$, defined by:
\[
\tilde{\G}_i = \{A \subset [n]: (1 \in A) \wedge (A \cap \{2,3,\ldots,i\} \neq \emptyset)\} \cup \{A \subset [n]: (1 \not \in A) \wedge (\{2,3,\ldots,i \} \subset A)\}.
\]
Indeed, we have $\mu_p(\tilde{\G}_i) = p(1-(1-p)^{i-1})+(1-p)p^{i-1}$, which corresponds to condition~\eqref{Eq:Cor-Condition-intro} with $\epsilon=p^{i-1}$, and
$\mu_p(\tilde{\G}_i \setminus \F_1) = (1-p)p^{i-1}=(1-p)\epsilon$.

Note that any intersecting family $\F$ satisfies $\mu_{1/2}(\F) \leq 1/2$, since it cannot contain both a set and its complement. Hence, Theorem~\ref{Thm:Biased1-intro} immediately implies a stability version of the biased EKR theorem (i.e., of Theorem~\ref{Thm:Biased-EKR}). Moreover, this stability version is tight, since the families $\tilde{\G}_i$ are intersecting.

Hence, not only do we obtain a tight stability version of the biased EKR theorem; we also show that the result holds under a significantly weaker condition. Indeed, instead of requiring that the family is intersecting, it is sufficient to require that it is contained in an increasing family $\F'$ satisfying $\mu_{1/2}(\F') \leq 1/2$. This condition holds for many families that are far from satisfying any intersection property.

Furthermore, our stability result is strong, in the sense that for small $p$, it holds not only when $\mu_p(\F)$ is close to $p$ (the maximal possible value), but rather whenever $\mu_p(\F) > Cp^2$, as is apparent from condition~\eqref{eq:hypo-intro}.

Similarly, we obtain a strong stability version of Theorem \ref{thm:Wilson-biased} (the biased version of Wilson's theorem), one that holds not only for $t$-intersecting families, but also for any increasing family $\F \subset \p([n])$ satisfying $\mu_{1/(t+1)}(\F) \leq (t+1)^{-t}$.

\subsubsection*{Our proof-techniques in the $p$-biased setting}
Our general strategy for proving stability theorems in the $p$-biased setting is as follows. Given an increasing family $\F \subset \mathcal{P}([n])$, we view it as a subset of the hypercube $\{0,1\}^n$, and we compare the measures $\mu_p(\F)$ for different values of $p$. A well-known lemma of Russo \cite{Russo} states that the derivative of the function $f: p \mapsto \mu_p(\F)$ satisfies $\frac{df}{dp} = \mu_{p}(\partial \F)$, where $\partial \F$ denotes the edge boundary of $\F$. (If $S \subset \{0,1\}^n$, the {\em edge boundary} $\partial S$ of $S$ is defined, as usual, to be the set of edges of the hypercube that join an element of $S$ to an element of $\{0,1\}^n \setminus S$. We define $\mu_p(xy) := \mu_p(x)+\mu_p(y)$, for any hypercube edge $xy$.) Our assumptions on $\F$ supply us with two values $p_1<p_0$ such that $\mu_{p_0}(\F)$ is not much larger than $\mu_{p_1}(\F)$. (In Theorem \ref{Thm:Biased1-intro}, for example, $p_0=1/2$, and $p_1=p$.) By applying the Mean Value theorem to the function $f$, it follows that there exists $p_2 \in (p_1,p_0)$ such that the edge boundary of $\F$ is small with respect to $\mu_{p_2}$, i.e., $\mu_{p_2}(\partial \F)$ is small. (We note that the same idea was used in a slightly different context, under weaker conditions on $|\mu_{p_0}(\F)-\mu_{p_1}(\F)|$, by Dinur and Safra in \cite{Dinur-Safra}.)

On the other hand, the biased version of the edge-isoperimetric inequality on the hypercube (Theorem \ref{thm:skewed-iso} below; see e.g.\ Kahn and Kalai \cite{Kahn-Kalai}) asserts that for any $p \in (0,1)$ and any increasing $\mathcal{F} \subset \p([n])$, $\mu_p(\partial \F )$ cannot be too small as function of $\mu_p(\F)$, and the minimum is attained (for any $p$) by $t$-umvirates alone. Furthermore, a recent stability version of this isoperimetric inequality (due to the authors \cite{EKL16+}) implies that if $\mu_p(\partial \F)$ is small, then $\F$ is close to a $t$-umvirate (with respect to $\mu_p$). Applying this fact with $p=p_2$, where $p_2$ is obtained from the argument in the preceding paragraph, we see that the original increasing family $\F$ is close to a $t$-umvirate with respect to $\mu_{p_2}$. A monotonicity argument implies that $\F$ is close to the same $t$-umvirate with respect to $\mu_{p_1}$.

The argument sketched above yields a rough stability result in the biased-measure setting. By considering more carefully the relation between the part of $\F$ inside the approximating $t$-umvirate and the part of $\F$ outside, and using a monotonicity argument (relying once again on the biased edge-isoperimetric inequality for the hypercube, together with the fact that $\F$ is increasing), we are able to bootstrap our rough stability result into a much sharper stability result in the biased-measure setting, yielding Theorem \ref{Thm:Main}, our main biased stability theorem.

A stability version of Theorem \ref{thm:Wilson-biased} (the biased version of Wilson's theorem) follows almost immediately. Indeed, given a $t$-intersecting family $\F \subset \p([n])$ such that $\mu_p(\F)$ large (for some $p < 1/(t+1)$), we consider its up-closure $\F^{\uparrow} : = \{S \subset [n]:\ S \supset F \text{ for some }F \in \F\}$; this is an increasing family, and since it is $t$-intersecting we have $\mu_{1/(t+1)}(\F^{\uparrow}) \leq (t+1)^{-t}$. Hence, $\F^{\uparrow}$ satisfies the hypotheses of our main biased stability theorem, which implies that $\F^{\uparrow}$ is close to a $t$-umvirate with respect to $\mu_p$. The same must hold for $\F$, completing the proof.

We also obtain a dual version (Theorem \ref{Thm:Dual}) of our biased stability theorem, and we use this to obtain a stability result for the biased version of the Erd\H{o}s matching conjecture.

\subsection*{Passing to the $k$-uniform setting}
With our biased stability theorems in hand, we move on to consider $k$-uniform families, in Section~\ref{sec:uniform}. We first use our main biased stability theorem to obtain a rough stability result in the $k$-uniform setting. For this, we use a variant of a method of Dinur and Friedgut \cite{DF09}, who observed that if $\F \subset [n]^{(k)}$, where $k/n$ is bounded away from $0$ and $1$, then the `local LYM inequality' (sometimes known as the `weak Kruskal-Katona theorem'; see e.g.\ \cite{Bollobas} \S 3) implies that $|\F| / {n \choose k} \lesssim \mu_p(\F^{\uparrow})$, where $p = k/n+o(1)$. In other words, the uniform measure of a $k$-uniform family is well-approximated by the $p$-biased measure of its up-closure, if $p$ is a little larger than $k/n$. We show that if $\F \subset [n]^{(k)}$ has size close to a $k$-uniform $t$-umvirate, then $\F^{\uparrow}$ has $p$-biased measure not too far below that of a $t$-umvirate in $\p([n])$, if $p$ is a little larger than $k/n$. (This argument relies upon the full Kruskal-Katona theorem \cite{Katona66,Kruskal63}, for which a $t$-umvirate is fortunately extremal.) If $\F$ is $t$-intersecting, then so is $\F^{\uparrow}$, and therefore by Theorem \ref{thm:Wilson-biased}, we have $\mu_{1/(t+1)}(\F^{\uparrow}) \leq (t+1)^{-t}$. Hence, we can apply our main biased stability theorem to $\F^{\uparrow}$, deducing that $\F^{\uparrow}$ is close to a $t$-umvirate in $\p([n])$. Another application of the Kruskal-Katona theorem implies that the original family $\F$ was close to a $k$-uniform $t$-umvirate.

The above argument yields a rough stability result in the $k$-uniform setting. To bootstrap this to the much more precise Theorem \ref{thm:stability-uniform-t}, we prove several technical lemmas on cross-intersecting families (again in the uniform setting), using some intricate combinatorial arguments. We use these lemmas to compare the part of $\F$ inside the approximating $t$-umvirate with the part of $\F$ outside, splitting up the latter according to the intersection of sets with the $t$-element set determining the $t$-umvirate.

\subsection*{Organization of the paper}

In Section \ref{sec:methods}, we give some definitions and notation, and present some of the main tools and prior results used in our paper. In Section~\ref{sec:proof}, we prove our main theorems in the biased-measure setting and present several applications. In Section~\ref{sec:uniform}, we prove our stability results in the classical setting of $k$-element sets. In Section~\ref{sec:intersecting}, we
compare our results with some prior results on intersecting and $t$-intersecting families. We conclude the paper with some open problems, in Section~\ref{sec:open-problems}.

\section{Definitions and tools}
\label{sec:methods}
\subsection{Definitions and notation}
If $\mathcal{F} \subset \p([n])$, we define the {\em dual} family $\F^*$ by $\mathcal{F}^* = \{[n] \setminus A:\ A \notin \F\}$. We write $\mathcal{F}^c := \p([n]) \setminus \F$ and $\overline{\mathcal{F}} := \{[n] \setminus A:\ A \in \F\}$, so that $\mathcal{F}^* = (\bar{\mathcal{F}})^c = \overline{\mathcal{F}^c}$. We denote by $\mathcal{F}^{\uparrow}$ the up-closure of $\mathcal{F}$, i.e. the minimal increasing subfamily of $\p([n])$ which contains $\mathcal{F}$. If $\F \subset \p([n])$ and $C \subset B \subset [n]$, we define $\F_{B}^{C}:= \left\{S\in\mathcal{P}\left(\left[n\right]\setminus B\right)\,:\, S\cup C \in \mathcal{F}\right\}$. If $B \subset [n]$, we write $\mathcal{S}_B := \{A \subset [n]:\ B \subset A\}$, and we write $\OR_{B} := (\mathcal{S}_B)^{*} = \{A \subset [n]: A \cap B \neq \emptyset\}$.

Let $Q_n$ denote the Hamming graph of the $n$-dimensional hypercube, i.e. the graph with vertex-set $\{0,1\}^n$, where two vertices are joined by an edge if they differ in exactly one coordinate. If $A \subset \{0,1\}^n$, we let $\partial A$ denote the {\em edge-boundary} of the set $A$ w.r.t. $Q_n$, meaning the set of edges of $Q_n$ which join a vertex in $A$ to a vertex in $\{0,1\}^n \setminus A$. We will often identify $\{0,1\}^n$ with $\mathcal{P}([n])$, via the correspondence $(x_1,x_2,\ldots,x_n) \leftrightarrow \{i \in [n]:\ x_i=1\}$. Hence, we say $A \subset \{0,1\}^n$ is increasing if the corresponding subset of $\mathcal{P}([n])$ is increasing, i.e. if whenever $(x_1,\ldots,x_n) \in A$ and $y_i \geq x_i$ for all $i \in [n]$, we have $(y_1,\ldots,y_n) \in A$.

We write $1_{A}$ for the {\em indicator function} of $A$, i.e., the Boolean function
$$1_{A}:\ \{0,1\}^n \to \{0,1\};\quad 1_{A}(x) = \begin{cases} 1 & \textrm{ if } x \in A;\\ 0 & \textrm{ if } x \notin A.\end{cases}$$

A {\em subcube} of $\{0,1\}^n$ is a set of the form $\{x \in \{0,1\}^n:\ x_i = c_i\ \forall i \in T\}$, where $T \subset [n]$ and $c_i \in \{0,1\}$ for each $i \in T$. Under the identification above, a $t$-umvirate corresponds to an $(n-t)$-dimensional, increasing subcube.

For each $i \in [n]$, we let $e_i = (0,0,\ldots,0,1,0,\ldots,0)$ denote the $i$th unit vector. For $x,y \in \{0,1\}^n$, we let $x+y$ denote the sum of $x$ and $y$ modulo 2.

Let $0 \leq p \leq 1$. By identifying $\{0,1\}^n$ with $\p([n])$ as above, the $p$-biased measure on $\p([n])$ (defined in the introduction) can alternatively be defined on $\{0,1\}^n$:
$$\mu_p(A) = \sum_{x \in A} p^{|\{i \in [n]:\ x_i=1\}|} (1-p)^{|\{i \in [n]:\ x_i=0\}|}\quad \forall A \subset \{0,1\}^n.$$

As usual, if $f:\{0,1\}^n \to \mathbb{R}$, we let $\mu_p[f]$ denote the expectation of $f$ with respect to the measure $\mu_p$.

We say that a function $f:\{0,1\}^n \to \mathbb{R}$ is {\em increasing} if whenever $x,y \in \{0,1\}^n$ with $x_i \leq y_i$ for all $i \in [n]$, we have $f(x) \leq f(y)$.

If $f:\{0,1\}^n \to \{0,1\}$ is a Boolean function, the {\em influence of $f$ in direction $i$} (with respect to $\mu_p$) is defined by
$$\Inf_{i}[f] := \mu_p(\{x \in \{0,1\}^n:\ f(x) \neq f(x+e_i)\}).$$
The {\em total influence} of $f$ (w.r.t. $\mu_p$) is $I_p[f] := \sum_{i=1}^{n} \Inf_{i}[f]$.

Similarly, if $A \subset \{0,1\}^n$, the {\em influence of $A$ in direction $i$} (w.r.t. $\mu_p$) is defined by $\Inf_i[A] : = \Inf_i[1_{A}]$, and the {\em total influence} of $A$ (w.r.t. $\mu_p$) is $I_p[A] : = I_p[1_{A}]$.

Throughout the paper, we use the letters $C,c$ to denote positive constants (possibly depending upon some parameters); the values of these constants and the parameters on which they are allowed to depend may differ between different theorems, propositions etc.

We use the convention that ${a \choose b} = 0$ if $a,b \in \mathbb{Z}$ with $a <0$ or $b <0$; moreover, ${0 \choose 0}=1$ and ${0 \choose b} = 0$ for all $b > 0$.

\subsection{Background and tools}
\subsubsection*{The edge-isoperimetric inequality on the hypercube}

Harper \cite{Harper}, Lindsay \cite{lindsay}, Bernstein \cite{bernstein} and Hart \cite{hart} solved the edge-isoperimetric problem for $Q_n$, showing that among all subsets of $\{0,1\}^n$ of fixed size, initial segments of the binary ordering on $\{0,1\}^n$ have the smallest edge-boundary. (The {\em binary ordering} on $\{0,1\}^n$ is defined by $x < y$ iff $\sum_{i=1}^{n} 2^{i}x_i < \sum_{i=1}^{n} 2^{i}y_i$.) The following weaker (but still useful) statement has an easy proof by induction on $n$.
\begin{theorem}
\label{thm:harper}
If $A \subset \{0,1\}^n$, then
\begin{equation} \label{eq:harper} |\partial A| \geq |A| \log_2 (2^n/|A|). \end{equation}
Equality holds in (\ref{eq:harper}) if and only if $A$ is a subcube.
\end{theorem}

The following analogue of Theorem \ref{thm:harper} for the $p$-biased measure is well-known; it is stated and proved by Kahn and Kalai in \cite{Kahn-Kalai}, for example.
\begin{theorem}
\label{thm:skewed-iso}
Suppose that $0 < p < 1$ and $A \subset \{0,1\}^n$ is increasing, or alternatively that $0 < p \leq 1/2$ and $A \subset \{0,1\}^n$ is arbitrary. Then
\begin{equation} \label{eq:skewed-iso} pI_p[A] \geq \mu_p(A) \log_p (\mu_p(A)). \end{equation}
If $0 < p < 1$ and $A$ is increasing, then equality holds in (\ref{eq:skewed-iso}) if and only if $A$ is an increasing subcube (i.e., a $t$-umvirate for some $t \in \mathbb{N}$).
\end{theorem}
Like Theorem \ref{thm:harper}, this can be proved by a straightforward induction on $n$. We note that (\ref{eq:skewed-iso}) does not hold for arbitrary subsets $A \subset \{0,1\}^n$ when $p > 1/2$; indeed, the antidictatorship $A = \{x \in \{0,1\}^n:\ x_1=0\}$ is a counterexample.


\subsubsection*{A stability version of the biased edge-isoperimetric inequality on the hypercube}

In \cite{Ellis11}, the first author proved the following.
\begin{theorem}
\label{thm:almost-iso}
There exists an absolute constant $c_0 >0$ such that the following holds. Let $0 \leq \epsilon \leq c_0$. If $A \subset \{0,1\}^n$ with $|\partial A| \leq |A| (\log_2(2^n/|A|)+\epsilon)$, then there exists a subcube $C \subset \{0,1\}^n$ such that
$$|A \Delta C| \leq \frac{2\epsilon}{\log_2(1/\epsilon)}|A|.$$
\end{theorem}
Theorem~\ref{thm:almost-iso} says that a subset of $\{0,1\}^n$ whose edge-boundary has size close to the bound (\ref{eq:harper}) must be close in symmetric difference to a subcube. The proof uses a lower bound of Talagrand~\cite{Talagrand94} concerning the vector of influences, obtained by Fourier-analytic techniques.

In this paper, we need the following $p$-biased analogue of Theorem \ref{thm:almost-iso} for increasing Boolean functions, which we proved in \cite{EKL16+}.

\begin{theorem}
\label{thm:0.99}[\cite{EKL16+}, Theorem 1.8]
Let $\eta>0$. There exist $C_1 = C_1(\eta),\ c_0 = c_0(\eta)>0$ such that the following holds. Let $0<p\leq 1-\eta$, and let $0 \leq \epsilon \leq c_{0}/\ln(1/p)$. Let $f\colon\left\{ 0,1\right\} ^{n}\to\left\{ 0,1\right\}$ be an increasing Boolean function such that
$$pI_p[f] \leq \mu_{p}[f] \left(\log_{p}\left(\mu_{p}[f] \right)+\epsilon\right).$$
Then $f$ is $\frac{C_1\epsilon\ln(1/p)}{\ln\left(\frac{1}{\epsilon\ln(1/p)}\right)}\mu_p[f]$-
close (w.r.t. $\mu_p$) to the indicator function of an increasing subcube.
\end{theorem}

Our proof of Theorem~\ref{thm:0.99} in~\cite{EKL16+} follows a similar approach to the proof of Theorem \ref{thm:almost-iso} in \cite{Ellis11}, but is somewhat simpler and uses only elementary techniques (it does not rely on any result proved using Fourier analysis, for example).

\subsubsection*{Russo's Lemma and a combination with the edge-isoperimetric inequality}

Russo's Lemma \cite{Russo} relates the derivative of the function $p \mapsto \mu_p(A)$ to the total influence $I_p(A)$, where $A \subset \{0,1\}^n$ is increasing.
\begin{lemma}[Russo's Lemma]\label{Lemma:Russo1}
Let $A \subset \{0,1\}^n$ be increasing, and let $0 < p_0 < 1$. Then
$$\frac{d \mu_p(A)}{dp} \Big|_{p=p_0} = I_{p_0}[A].$$
\end{lemma}

\noindent We will make crucial use of the following consequence of Russo's Lemma and Theorem \ref{thm:skewed-iso}.
\begin{lemma}
\label{lemma:mono-decreasing}
Let $A \subset \{0,1\}^n$ be increasing. Then the function $p \mapsto \log_p(\mu_p(A))$ is monotone non-increasing on $(0,1)$. If $A$ is not a subcube, then this function is strictly monotone decreasing on $(0,1)$.
\end{lemma}
This is well-known (see for example \cite{Grimmett}, Theorem 2.38); for completeness, we provide a proof.
\begin{proof}[Proof of Lemma \ref{lemma:mono-decreasing}.]
Let $f(p):= \log_p (\mu_p(A)) = \ln \mu_p(A) / \ln p$. We have
\begin{equation}\label{eq:derivative}
\frac{df}{dp} = \frac{\frac{1}{\mu_p(A)} \frac{d \mu_p(A)}{d p} \ln p - \frac{1}{p}\ln (\mu_p(A)) }{(\ln p)^2}
= \frac{\frac{1}{\mu_p(A)} I_p[A] \ln p - \frac{1}{p}\ln (\mu_p(A)) }{(\ln p)^2} \leq 0,
\end{equation}
using Russo's Lemma and (\ref{eq:skewed-iso}). If $A$ is not a subcube, then strict inequality holds in (\ref{eq:skewed-iso}) for all $p \in (0,1)$, so strict inequality holds in (\ref{eq:derivative}), and therefore $f$ is strictly monotone decreasing.
\end{proof}

\noindent The following useful monotonicity lemma is an easy consequence of Lemma~\ref{lemma:mono-decreasing}.
\begin{lemma}\label{lemma:measures-decrease}
Let $p_{0}\in\left(0,1\right)$, $p\in (0,p_{0})$ and $t>0$. Let $\F \subset \p([n])$ be an increasing family.
\begin{enumerate}
\item Suppose $\mu_{p_{0}}\left(\mathcal{F}\right)\le p_{0}^{t}$. Then $\mu_{p}\left(\mathcal{F}\right)\le p^{t}$, with equality if and only if $t \in \mathbb{N}$ and $\mathcal{F}$ is a $t$-umvirate.

\item Suppose $\mu_{p_{0}}\left(\mathcal{F}\right)\le1-\left(1-p_{0}\right)^{t}$. Then $\mu_{p}\left(\mathcal{F}\right)\le1-\left(1-p\right)^{t}$, with equality if and only if $t \in \mathbb{N}$ and $\mathcal{F} = \OR_B$ for some $B \in [n]^{(t)}$.
\end{enumerate}
\end{lemma}

\begin{proof}
The upper bound in item~(1) follows immediately from Lemma~\ref{lemma:mono-decreasing}. For the equality part, suppose that $\mathcal{F} \subset \mathcal{P}([n])$ is not an increasing subcube. Then  the function $f$ defined in the proof of Lemma \ref{lemma:mono-decreasing} (of course, setting $A:=\mathcal{F}$) is strictly decreasing, and therefore $\mu_p(\mathcal{F}) < p^t$ for any $p \in (0,p_0)$. Hence, if $\mu_p(\mathcal{F}) = p^t$, then $\mathcal{F}$ must be an increasing subcube; it must therefore be a $t$-umvirate.

To prove item~(2), let $\F^*$ be the dual of $\F$. Then
\[
\mu_{1-p_{0}}(\F^*)=\mu_{p_{0}}\left(\mathcal{F}^{c}\right) \ge \left(1-p_{0}\right)^{t}.
\]
Since $\F^*$ is increasing, we may apply item~(1) for $1-p \geq 1-p_{0}$, obtaining
\[
1-\mu_{p}(\F) = \mu_p(\F^c)=\mu_{1-p}(\F^*)\ge\left(1-p\right)^{t}.
\]
The assertion follows by rearranging. The equality part follows exactly as in item (1).
\end{proof}

Note that the $t=1$, $p_0=1/2$ case of Lemma \ref{lemma:measures-decrease} (1), immediately yields the following strengthening of the biased EKR theorem:
\begin{corollary}
Let $\mathcal{F} \subset \p([n])$ be an increasing family such that $\mu_{1/2}(\F) \leq 1/2$. Then for any $p \in (0,1/2)$, $\mu_p(\mathcal{F}) \leq p$. Equality holds if and only if $\mathcal{F}$ is a dictatorship.
\end{corollary}

It is perhaps somewhat surprising that the conclusion of the biased EKR theorem holds under a considerably weaker hypothesis than being intersecting.
\medskip

We will also need the following consequence of Lemma \ref{lemma:mono-decreasing}.
\begin{lemma}\label{Lemma:Cross-Intersecting}
Let $0 < p \leq 1/2$. Suppose $\F,\G \subset \mathcal{P}([n])$ are increasing and cross-intersecting (meaning that $A\cap B \neq \emptyset$ whenever $A \in \F$ and $B \in \G$). Then
$$\mu_p(\G) \leq (1-\mu_p(\F))^{\log_{1-p}(p)}.$$
\end{lemma}
\begin{proof}
Since $\F$ and $\G$ are cross-intersecting, we have $\G \subset \F^{*}$. Hence, $\mu_{1-p}(\G) \leq \mu_{1-p}(\F^*) = 1-\mu_{1-p}(\bar{\F}) = 1-\mu_p(\F)$. Hence, by Lemma \ref{lemma:mono-decreasing}, $\mu_p(\G) \leq (\mu_{1-p}(\G))^{\log_{1-p}(p)} \leq (1-\mu_p(\F))^{\log_{1-p}(p)}$.
\end{proof}

\subsubsection*{Shadows of $t$-intersecting families}

The \emph{lower shadow} of a family $\F \subset [n]^{(k)}$ is defined as $\partial(\F) = \{A \in [n]^{(k-1)}: \exists B \in \F, A \subset B\}$. Similarly,
the \emph{$s$-shadow} of $\F$ is defined as $\partial^s(\F) = \{A \in [n]^{(k-s)}: \exists B \in \F, A \subset B\}$. The shadow is a central notion in extremal combinatorics that appears in a multitude of techniques and results (e.g., the classical Kruskal-Katona theorem~\cite{Katona66,Kruskal63}).

As we replace families of $k$-element subsets with increasing families, we define a notion of shadow for them. For an increasing family $\F \subset \p([n])$, we define
the $s$-shadow of $\F$ by
$$\partial^s(\F) := \{A \in \p([n]): \exists C \in [n]^{(s)}, A \cup C \in \F\}.$$
Note that in this definition we do {\em not} require $A \cap C = \emptyset$, so our definition of the $s$-shadow of an increasing family is weaker than the above definition of the $s$-shadow of a $k$-uniform family; however, it is the right notion for our purposes. We use the same notation for each notion; the one we mean will always be clear from the context.

We use the following classical theorem of Katona~\cite{Katona64} on $t$-shadows of $t$-intersecting subsets of $[n]^{(k)}$.
\begin{theorem}[Katona's shadow/intersection theorem, 1964]\label{Thm:Katona}
Let $\mathcal{F}\subset [n]^{(k)}$ be a $t$-intersecting family. Then $\left|\partial^{t}\left(\mathcal{F}\right)\right|\ge |\mathcal{F}|$.
\end{theorem}
\noindent We also use the following analogue of Theorem~\ref{Thm:Katona} for increasing subsets of $\p([n])$.
\begin{theorem}\label{Thm:Katona-Our}
Let $\mathcal{F}\subset \mathcal{P}\left(\left[n\right]\right)$ be an increasing $t$-intersecting family. Then for any $0<p<1$,
\[
\mu_{p}\left(\partial^{t}\left(\mathcal{F}\right)\right)\ge\frac{\left(1-p\right)^{t}}{p^{t}}\mu_{p}\left(\mathcal{F}\right).
\]
\end{theorem}

\begin{proof}
We deduce this from Theorem \ref{Thm:Katona} using the Dinur-Safra / Frankl-Tokushige method of `going to infinity and back' (see \cite{Dinur-Safra,FT03}). For a family $\mathcal{F}\subset \mathcal{P}\left(\left[n\right]\right)$
and $N>n$, we define
$$\mathcal{F}_{N,k} : = \{A \in [N]^{(k)}:\ A \cap [n] \in \mathcal{F}\}.$$
We claim that
\begin{equation}\label{Eq:Katona-Our1}
\partial^t(\mathcal{F}_{N,k}) \subset (\partial^t(\mathcal{F}))_{N,k-t}.
\end{equation}
Indeed, if $A \in \partial^t(\mathcal{F}_{N,k})$, then there exists $C \in [N]^{(t)}$ such that $A \cap C = \emptyset$ and $A \cup C \in \mathcal{F}_{N,k}$. Hence, $(A \cap [n]) \cup (C \cap [n]) = (A \cup C) \cap [n] \in \mathcal{F}$. Since $\mathcal{F}$ is increasing, for any $C' \in [n]^{(t)}$ with $C \cap [n] \subset C'$, we have $(A \cap [n]) \cup C' \in \mathcal{F}$. Therefore, $A \cap [n] \in \partial^t(\mathcal{F})$, and thus, $A \in (\partial^t(\mathcal{F}))_{N,k-t}$.

\mn As $\mathcal{F} \subset \mathcal{P}([n])$ is increasing and $t$-intersecting, for any $N> k > n$, $\mathcal{F}_{N,k}$ is also $t$-intersecting. Hence, by Theorem \ref{Thm:Katona} we have $|\partial^t(\mathcal{F}_{N,k})| \geq |\mathcal{F}_{N, k}|$. Combining this with~\eqref{Eq:Katona-Our1}, we get
$$|\mathcal{F}_{N,k}| \leq |\partial^t(\mathcal{F}_{N,k})| \leq |(\partial^t(\mathcal{F}))_{N,k-t}|.$$
Now, it is easily checked that for any $\mathcal{F} \subset \mathcal{P}([n])$,
$$\mu_p(\mathcal{F}) = \lim_{N \to \infty} \frac{|\mathcal{F}_{N, \lfloor pN \rfloor}|}{{N \choose \lfloor pN \rfloor}} = \lim_{N \to \infty} \frac{|\mathcal{F}_{N, \lfloor pN \rfloor-t}|}{{N \choose \lfloor pN \rfloor-t}}.$$
Hence,
\begin{align*}
\mu_p(\partial^t(\mathcal{F})) & = \lim_{N \to \infty} \frac{|(\partial^t (\mathcal{F}))_{N, \lfloor pN \rfloor-t}|}{{N \choose \lfloor pN \rfloor-t}} \geq \lim_{N \to \infty} \frac{|\mathcal{F}_{N, \lfloor pN \rfloor}|}{{N \choose \lfloor pN \rfloor-t}}
= \left(\lim_{N \to \infty}  \frac{{N \choose \lfloor pN \rfloor}}{{N \choose \lfloor pN \rfloor-t}}\right) \left(\lim_{N \to \infty}  \frac{|\mathcal{F}_{N, \lfloor pN \rfloor}|}{{N \choose \lfloor pN \rfloor}}\right)\\
&= \left(\frac{1-p}{p}\right)^t \lim_{N \to \infty}  \frac{|\mathcal{F}_{N, \lfloor pN \rfloor}|}{{N \choose \lfloor pN \rfloor}} =
\left(\frac{1-p}{p}\right)^t \mu_p(\mathcal{F}),
\end{align*}
as asserted.
\end{proof}

\section{The main `biased' stability theorem}
\label{sec:proof}

In this section, we state and prove our main biased-measure stability theorem, and present several of its applications.
\begin{theorem}\label{Thm:Main}
Let $t \in \mathbb{N}$ and let $0 < p_0 < 1$. Then there exist $C=C(p_0,t)>2/p_0$ and $c=c(p_0,t)>0$ such that
the following holds. Let $0 < p < p_0$ and define $\tilde{c}:=\left(\frac{1-p_{0}}{p_{0}}\right)^{\log_{1-p_{0}}\left(1-p\right)}$. Let $\mathcal{F} \subset \p([n])$ be an increasing family such that $\mu_{p_{0}}\left(\mathcal{F}\right)\le p_{0}^t$ and
\begin{equation}\label{Eq:Condition0}
\mu_p(\F) \geq \begin{cases} p^t(1-c(p_0-p)) & \text{ if } p \geq 1/C\\ Cp^{t+1} & \text{ if } p < 1/C.\end{cases}
\end{equation}
Let $\epsilon >0$. If
\begin{equation}\label{Eq:Condition}
\mu_{p}\left(\mathcal{F}\right)\ge p^t\left(1-\tilde{c}\epsilon^{\log_{p}(p_{0}) \log_{1-p_{0}}(1-p)}\right)+\left(1-p\right)p^{t-1}\epsilon,
\end{equation}
then there exists a $t$-umvirate $\mathcal{S}_B$ such that
\begin{equation}\label{Eq:Assertion}
\mu_{p}\left(\mathcal{F} \setminus \mathcal{S}_B \right) \le (1-p)p^{t-1}\epsilon.
\end{equation}
\end{theorem}
\begin{remark}
It will follow from our proof that for any $\xi \in (0,1/2]$,
$$\sup_{p_0 \in [\xi,1-\xi]} C(p_0,t) = O_{\xi,t}(1),\quad \inf_{p_0 \in [\xi,1-\xi]} c(p_0,t) = \Omega_{\xi,t}(1).$$
\end{remark}

\noindent Theorem~\ref{Thm:Main} is tight for infinitely many of the families $\{\tilde{\h}_{t,s,r}:\ t,s,r \in \mathbb{N}\}$, defined by:
\begin{align*}
\tilde{\h}_{t,s,r} =& \{A \subset [n]: ([t] \subset A) \wedge (A \cap \{t+1,t+2,\ldots,t+r\} \neq \emptyset)\} \\
&\cup \{A \subset [n]: ([t-1] \subset A) \wedge (t \not \in A) \wedge (\{t+1,t+2,\ldots,t+s \} \subset A)\}.
\end{align*}
To see this, for each $(t,s,r) \in \mathbb{N}^3$ with $r,s \geq 2$, choose the unique $p_0 \in (0,1)$ such that $(1-p_0)^{r-1}=p_0^{s-1}$. Then we have $\mu_{p_0}(\tilde{\h}_{t,s,r}) = p_0^t$, $\mu_p(\tilde{\h}_{t,s,r}) = p^t(1-(1-p)^{r})+(1-p)p^{t+s-1}$, and $\mu_p(\tilde{\h}_{t,s,r} \setminus \s_{[t]}) = (1-p)p^{t+s-1}$, so equality holds in (\ref{Eq:Condition}) and (\ref{Eq:Assertion}) with $\epsilon=p^s$. Provided $r$ is sufficiently large, the condition (\ref{Eq:Condition0}) is also satisfied.

Roughly speaking, Theorem~\ref{Thm:Main} asserts that if $\mu_{p_0}(\F) \leq p_0^t$ (this replaces the
$t$-intersection condition) and $\mu_p(\F)$ is `large enough', then the following stability result
holds: if $\mu_p(\F) \geq p^t(1-\epsilon')$, then there exists a $t$-umvirate $\s_B$ such that
$\mu_p(\F \setminus \s_B) \lessapprox O(\epsilon'^{\log_{p_0}(p) \log_{1-p} (1-p_0)})$.
The meaning of `large enough', reflected by condition~\eqref{Eq:Condition0}, differs between
the cases of large $p$ and small $p$. When $p = \Omega(p_0)$, the stability theorem requires
that $\mu_p(\F)$ is close to $p^t$. For $p \ll p_0$, the
assertion is stronger, saying that the conclusion holds once $\mu_p(\F)$ is significantly larger than
$p^{t+1}$ (even if it is far from $p^t$).

\medskip Since the proof of the theorem is somewhat complex, we first present the three components of the proof separately, and then we show how combining them yields the theorem. As in~\cite{DF09,Friedgut08}, we distinguish between the cases of large $p$, i.e., $p=\Omega_{p_0}(1)$, and of small $p$, i.e., $p \ll p_0$, and handle each case differently.

First, in section \ref{sec:sub:basic-stab}, we prove a rough stability result for large $p$, using Russo's lemma (as outlined in section \ref{sec:sub:methods}) and our stability version of the biased edge-isoperimetric inequality on the hypercube (Theorem~\ref{thm:0.99}). Then, in section \ref{sec:sub:rough-small-p}, we use our rough stability result for large $p$, combined with the Dinur-Safra / Frankl-Tokushige method of `going to infinity and back' and the Kruskal-Katona theorem, to prove a rough stability result for small $p$. The conclusion of these two rough stability results is that a family $\F$ satisfying the hypotheses of Theorem \ref{Thm:Main} is `somewhat' close to a $t$-umvirate. In section \ref{sec:sub:bootstrapping}, we present a bootstrapping argument showing that if an increasing family $\F$ satisfies $\mu_{p_0}(\F) \leq p_0^t$ and is `somewhat' close to a $t$-umvirate, then it must be very close to that $t$-umvirate. Finally, in section \ref{sec:sub:proof}, we apply the bootstrapping argument to yield our theorem.

The following lemma will be used repeatedly.

\begin{lemma}
\label{lem:bootstrapping-General} Let $p_{0},\delta\in\left(0,1\right)$, let $t,n\in\mathbb{N}$, and let $0 < p < p_0$. Let $\mathcal{F} \subset \mathcal{P}\left(\left[n\right]\right)$ be an increasing family with
\[
\mu_{p}\left(\mathcal{F}\setminus\mathcal{S}_{\left[t\right]}\right)\ge\left(1-p\right)p^{t-1}\delta.
\]
Then:

\medskip \noindent {\rm (a)}
\[
\mu_{p_{0}}\left(\mathcal{F}\setminus\mathcal{S}_{\left[t\right]}\right)\ge \left(1-p_0\right)p_{0}^{t-1}\left(\delta^{\log_{p}p_{0}}\right).
\]
\medskip \noindent {\rm (b)}
If, in addition, $\mu_{p_0}(\F) \leq p_0^t$, then
\[
\mu_{p}\left(\mathcal{F}\cap\mathcal{S}_{\left[t\right]}\right)\le p^{t}\left(1-\tilde{c}\delta^{\log_{p}p_{0}\log_{1-p_0}(1-p)}\right),
\]
where $\tilde{c}:=\left(\frac{1-p_{0}}{p_{0}}\right)^{\log_{1-p_{0}}\left(1-p\right)} = \frac{1-p}{p_0^{\log_{1-p_0}(1-p)}}$.
\end{lemma}

\begin{proof}
To prove~(a), we note that the following equations hold.
\begin{align*}
\mu_{p}\left(\mathcal{F}\setminus\mathcal{S}_{\left[t\right]}\right) & =\sum_{i\in\left[t\right]}\left(1-p\right)p^{i-1}\mu_{p}\left(\mathcal{F}_{\left[i\right]}^{\left[i-1\right]}\right),\\
\mu_{p_{0}}\left(\mathcal{F}\setminus\mathcal{S}_{\left[t\right]}\right) & =\sum_{i\in\left[t\right]}\left(1-p_{0}\right)p_{0}^{i-1}\mu_{p_{0}}\left(\mathcal{F}_{\left[i\right]}^{\left[i-1\right]}\right).
\end{align*}
By Lemma~\ref{lemma:measures-decrease}(1),
\[
\mu_{p_{0}}\left(\mathcal{F}_{\left[i\right]}^{\left[i-1\right]}\right)\ge\mu_{p}\left(\mathcal{F}_{\left[i\right]}^{\left[i-1\right]}\right)^{\log_{p}p_{0}}.
\]
Therefore,
\begin{alignat*}{1}
\mu_{p_{0}}\left(\mathcal{F}\setminus\mathcal{S}_{\left[t\right]}\right) & \ge\sum_{i\in\left[t\right]}\left(1-p_{0}\right)p_{0}^{i-1}\mu_{p}\left(\mathcal{F}_{\left[i\right]}^{\left[i-1\right]}\right)^{\log_{p}p_{0}}\\
 & =\frac{\left(1-p_{0}\right)}{\left(1-p\right)^{\log_{p}p_{0}}}\sum_{i\in\left[t\right]}\left(\left(1-p\right)p^{i-1}\mu_{p}
 \left(\mathcal{F}_{\left[i\right]}^{\left[i-1\right]}\right)\right)^{\log_{p}p_{0}}\\
&\ge \frac{\left(1-p_{0}\right)}{\left(1-p\right)^{\log_{p}p_{0}}} \left(\mu_{p}\left(\mathcal{F}\setminus\mathcal{S}_{\left[t\right]}\right)\right)^{\log_{p}p_{0}} \ge\left(1-p_{0}\right)p_{0}^{t-1}\left(\delta^{\log_{p}p_{0}}\right),
\end{alignat*}
where the penultimate inequality holds since for any non-negative $a_1,\ldots,a_\ell$ and for any $0 \leq \beta \leq 1$, we have $(\sum a_i)^{\beta} \leq \sum a_i^{\beta}$, and the last inequality uses the assumption on $\mu_{p}\left(\mathcal{F}\setminus\mathcal{S}_{\left[t\right]}\right)$. This completes the proof of (a).

\medskip

\noindent To prove~(b), we note that~(a) and the assumption $\mu_{p_0}(\F) \leq p_0^t$ imply that
\[
\mu_{p_{0}}\left(\mathcal{F}_{\left[t\right]}^{\left[t\right]}\right)=\frac{\mu_{p_{0}}\left(
\mathcal{\mathcal{F}\cap\mathcal{S}}_{\left[t\right]}\right)}{p_{0}^{t}}\le1-\delta^{\log_{p}p_{0}}\frac{\left(1-p_{0}\right)}{p_{0}}.
\]
By Lemma~\ref{lemma:measures-decrease}(2), this implies:
\[
\mu_{p}\left(\mathcal{F}\cap S_{\left[t\right]}\right)=p^{t}\mu_{p}\left(\mathcal{F}_{\left[t\right]}^{\left[t\right]}\right)\le p^{t}\left(1-\tilde{c}\delta^{\log_{p}p_{0}\log_{1-p_0}(1-p)}\right),
\]
as asserted.
\end{proof}

\subsection{A rough stability result for large $p$}
\label{sec:sub:basic-stab}

In this subsection we consider the case $p \geq \zeta p_0$, for a constant $\zeta>0$ that will be chosen later. (Meanwhile, we state the results
in terms of $\zeta$.)

The following proposition shows that if $\mu_{p}(\mathcal{F})$ is sufficiently close to $p^t$, then $\mathcal{F}$ is somewhat close to
a $t$-umvirate. We will use it in the proof of Theorem~\ref{Thm:Main} as the basis of a bootstrapping process.

\begin{proposition}
\label{lem:Weak Stability-general}
For any $\eta \in (0,1)$, there exists $C=C(\eta)>0$ such that the following holds. Let $\zeta \in (0,1)$, let $\zeta p_{0} \leq p \leq p_{0} \leq 1-\eta$, let $t \in \mathbb{N}$, and let $0 < \epsilon \leq 1$. Let $\F \subset \p([n])$ be an increasing family such that $\mu_{p_{0}}\left(\mathcal{F}\right)\le p_{0}^{t}$ and $\mu_{p}\left(\mathcal{F}\right)\ge p^{t} (1-\epsilon'),$
where $\epsilon':=\frac{C\epsilon\left(p_0-p\right)}{\ln\left(\frac{1}{\zeta p_0}\right)t}$.

\noindent Then there exists $B \in [n]^{(t)}$ such that $\mu_{p}\left(\mathcal{F}\setminus\mathcal{S}_{B}\right)\le\epsilon p^{t-1}\left(1-p\right)$.
\end{proposition}

\begin{proof}
The proof of the proposition has three steps. First, we use Russo's lemma (i.e., Lemma~\ref{Lemma:Russo1}) to show that for some $p'\in (p_{0},p)$, the total influence $I_{p'}\left(\mathcal{F}\right)$ is very close to the total influence of a $t$-umvirate. Then, we use our stability version of the biased edge-isoperimetric inequality on the hypercube (Theorem~\ref{thm:0.99}) to deduce that there exists $B \in [n]^{(t)}$ such that $\mu_{p'}\left(\mathcal{F}\setminus\mathcal{S}_{B}\right)$ is small. Finally, we use Lemma~\ref{lem:bootstrapping-General} to complete the proof.

\medskip

\noindent By choosing $C$ to be sufficiently small depending upon $\eta$, we may assume throughout that $\epsilon' \leq 1/2$. By Russo's lemma, we have
\begin{align*}
\intop_{p}^{p_{0}}I_{p'}\left(\mathcal{F}\right)dp' & =\mu_{p_{0}}\left(\mathcal{F}\right)-\mu_{p}\left(\mathcal{F}\right)\le p_{0}^{t}-p^{t}+\epsilon' p^t,\\
\intop_{p}^{p_{0}}I_{p'}\left(\mathcal{S}_{\left[t\right]}\right)dp' & =\mu_{p_{0}}\left(\mathcal{S}_{\left[t\right]}\right)-\mu_{p}\left(\mathcal{S}_{\left[t\right]}\right)=p_{0}^{t}-p^{t}.
\end{align*}
Subtracting the equations, we get
\[
\intop_{p}^{p_{0}}\left(I_{p'}\left(\mathcal{F}\right)-I_{p'}\left(\mathcal{S}_{\left[t\right]}\right)\right)dp \le\epsilon' p^t.
\]
Hence, there exists $p'\in (p,p_0)$ such that
\begin{equation}\label{Eq:Weak-stability-general1}
I_{p'}\left(\mathcal{F}\right)-I_{p'}\left(\mathcal{S}_{\left[t\right]}\right)\le\frac{\epsilon' p^t}{p_0-p}.
\end{equation}
In addition, by the assumption on $\mu_p(\F)$ and $\mu_{p_0}(\F)$ and by Lemma~\ref{lemma:measures-decrease}, we have:
\begin{equation}\label{Eq:Weak-stability-general2}
p'^{t} \geq \mu_{p'}\left(\mathcal{F}\right)\ge\left(p^{t}\left(1-\epsilon'\right)\right)^{\log_{p}p'}\ge p'^{t}\left(1-\epsilon'\right).
\end{equation}
Using~\eqref{Eq:Weak-stability-general1},~\eqref{Eq:Weak-stability-general2}, and the fact that $I_{p'}(\mathcal{S}_{[t]}) = t(p')^{t-1}$, we obtain:
\begin{align}\label{Eq:Weak-stability-general3}
\begin{split}
p'I_{p'}\left(\mathcal{F}\right) & \le p'\left(tp'^{t-1}+\frac{\epsilon' p^t}{p_{0}-p}\right) \leq
p'^{t}\left(\log_{p'}\left(\mu_{p'}\left(\mathcal{F}\right)\right)+\frac{\epsilon' p^t}{\left(p_0-p\right)p'^{t-1}}\right)\\
 & \le\mu_{p'}\left(\mathcal{F}\right)\left(\frac{\log_{p'}\left(\mu_{p'}\left(\mathcal{F}\right)\right)+\frac{\epsilon' p^t}{\left(p_0-p\right)p'^{t-1}}}{1-\epsilon'}\right) \\
 &=\mu_{p'}\left(\mathcal{F}\right)\left(\log_{p'}\left(\mu_{p'}\left(\mathcal{F}\right)\right)+\frac{\epsilon'\log_{p'} \left(\mu_{p'}\left(\mathcal{F}\right)\right)+\frac{\epsilon' p^t}{\left(p_0-p\right)p'^{t-1}}}{1-\epsilon'}\right)\\
 & \leq \mu_{p'}\left(\mathcal{F}\right)\left(\log_{p'}\left(\mu_{p'}\left(\mathcal{F}\right)\right)+2\epsilon'\log_{p'} \left(\mu_{p'}\left(\mathcal{F}\right)\right)+\frac{2\epsilon' p^t}{\left(p_0-p\right)p'^{t-1}}\right),
\end{split}
\end{align}
using the fact that $\epsilon' \leq 1/2$.

\medskip \noindent We claim that
\begin{equation}\label{Eq:Weak-stability-general4}
\epsilon'\log_{p'}\left(\mu_{p'}\left(\mathcal{F}\right)\right)+ \frac{\epsilon' p^t}{\left(p_0-p\right)p'^{t-1}} \le\frac{3C\epsilon}{\ln\left(\frac{1}{\zeta p_0}\right)}.
\end{equation}
Indeed, the definition of $\epsilon'$ immediately implies that $\frac{\epsilon' p^t}{\left(p_0-p\right)p'^{t-1}}\le\frac{C\epsilon}{\ln\left(\frac{1}{\zeta p_0}\right)}$. Using the fact that $\mu_{p}\left(\mathcal{F}\right)\ge p^t(1-\epsilon') \geq p^{t+1}$ and Lemma~\ref{lemma:measures-decrease}(1), we have
$\mu_{p'}\left(\mathcal{F}\right)\ge p'^{t+1}$, and thus, $\epsilon'\log_{p'}\left(\mu_{p'}\left(\mathcal{F}\right)\right)\le\left(t+1\right)\epsilon' \le\frac{2C\epsilon}{\ln\left(\frac{1}{\zeta p_0}\right)}$, implying~\eqref{Eq:Weak-stability-general4}.

\medskip \noindent Combining~\eqref{Eq:Weak-stability-general3} and~\eqref{Eq:Weak-stability-general4}, we get:
\[
p'I_{p'}\left(\mathcal{F}\right)\le\mu_{p'}\left(\mathcal{F}\right)\left(\log_{p'}\left(\mu_{p'}\left(\mathcal{F}\right)\right) +\frac{6C\epsilon}{\ln\left(\frac{1}{\zeta p_0}\right)}\right).
\]
Therefore, by Theorem \ref{thm:0.99} (and using the assumption $p \geq \zeta p_0$), there exists $B \subset [n]$ such that
\begin{equation}\label{Eq:Weak-stability-general5}
\mu_{p'}\left(\mathcal{F}\Delta\mathcal{S}_{B}\right)\le O_{\eta}\left(\frac{C\epsilon}{\ln\left(\frac{1}{C\epsilon}\right)}\right)\mu_{p'}\left(\mathcal{F}\right)\le p'^{t}O_{\eta} \left(\frac{C\epsilon}{\ln\left(\frac{1}{C\epsilon}\right)}\right)\le \tfrac{1}{2}p'^{t}\left(1-p'\right)\epsilon
\end{equation}
(the last inequality using the fact that $C$ is sufficiently small depending on $\eta$). Note that $|B|=t$. Indeed, if $|B| \leq t-1$, then
$$\mu_{p'}(\mathcal{F} \Delta \mathcal{S}_{B}) \geq \mu_{p'} (\mathcal{S}_{B}) - \mu_{p'}(\mathcal{F}) \geq (p')^{t-1} - (p')^t = (p')^{t-1}(1-p') > p'^{t}\left(1-p'\right)\epsilon,$$
contradicting (\ref{Eq:Weak-stability-general5}). On the other hand, if $|B| \geq t+1$, then provided $C$ is sufficiently small depending on $\eta$, we have $\epsilon' < \tfrac{1}{2}(1-p')$, and therefore
\begin{align*}
\mu_{p'}(\mathcal{F} \Delta \mathcal{S}_{B}) &\geq \mu_{p'}(\mathcal{F}) - \mu_{p'} (\mathcal{S}_{B})
\geq (p')^{t}(1-\epsilon') - (p')^{t+1}
 = (p')^{t}(1-p'-\epsilon') > \tfrac{1}{2}(p')^t (1-p') \\
 &\geq \tfrac{1}{2}(p')^t(1-p')\epsilon,
\end{align*}
again contradicting (\ref{Eq:Weak-stability-general5}).

\medskip \noindent Clearly, (\ref{Eq:Weak-stability-general5}) implies
\begin{equation}\label{Eq:Weak-stability-general6}
\mu_{p'}\left(\mathcal{F}\setminus \mathcal{S}_{B}\right)\le p'^{t}\left(1-p'\right)\epsilon.
\end{equation}

\noindent Finally, by Lemma~\ref{lem:bootstrapping-General}(a), Equation~\eqref{Eq:Weak-stability-general6} implies
\[
\mu_{p}\left(\mathcal{F\setminus\mathcal{S}_{B}}\right)\le p^{t-1}\left(1-p\right)\left(p'\epsilon\right)^{\log_{p'}p}\le p^{t}\epsilon(1-p) \leq (1-p)p^{t-1}\epsilon.
\]
This completes the proof.
\end{proof}

\begin{remark}
We note that after showing that the total influence $I_{p'}\left(\mathcal{F}\right)$ is very close to the total influence of a $t$-umvirate, we can apply Fourier-theoretic tools to show that $\F$ is close to a $t$-umvirate, instead of using Theorem \ref{thm:0.99}. Specifically, for $t=1$ we can obtain the assertion of Proposition~\ref{lem:Weak Stability-general} using a theorem of Nayar \cite{Nayar13}, which is a $p$-biased version of the Friedgut-Kalai-Naor (FKN) theorem~\cite{FKN}. For $t>1$ one can try to use the higher-degree, $p$-biased analogue of the FKN theorem due to Kindler and Safra~\cite{KS03} (as in~\cite{Friedgut08}), but this yields a weaker statement than Proposition~\ref{lem:Weak Stability-general}. Theorem \ref{thm:0.99} seems to us to be the right tool to use in the context.
\end{remark}

\subsection{A rough stability result for small $p$}
\label{sec:sub:rough-small-p}
In this subsection we consider the case of a small $p$, i.e., $p$ smaller than some constant depending only on $p_0$. We show that for a sufficiently
small $p$, a stability result can be obtained not only when $\mu_p(\mathcal{F})$ is close to $p^t$, but also under the weaker assumption
$\mu_p(\mathcal{F}) \geq C p^{t+1}$ for some $C=C(p_0,t)$. Our aim is to prove the following.

\begin{proposition}
\label{prop:small-p-rough}
For each $p_0 \in (0,1)$, $t \in \mathbb{N}$ and $\delta>0$, there exists $C = C(\delta,p_{0},t)>0$ such that
the following holds. Let $0 < p<p_{0}$. Let $\mathcal{F} \subset \p([n])$ be an increasing family such that
$\mu_{p_0}(\mathcal{F}) \leq p_{0}^{t}$, and such that $\mu_p(\mathcal{F})\ge Cp^{t+1}$. Then there exists $B \in [n]^{(t)}$ such that
$$\mu_{p}(\mathcal{F} \setminus \mathcal{S}_{B}) \leq \delta^{\log_{(p_0/2)}(p)} p^{t-1}(1-p).$$
\end{proposition}
For the proof, we introduce a few more definitions.

\begin{definition}
We say that $\mathcal{F} \subset \p([n])$ is {\em lexicographically ordered} if it is an initial segment of the lexicographic order on $\p([n])$, i.e. whenever $A \in \mathcal{F}$ and $B \subset [n]$ with $\min(A \Delta B) \in B$, we have $B \in \mathcal{F}$. Similarly, we say that $\mathcal{A} \subset [n]^{(k)}$ is {\em lexicographically ordered} if it is an initial segment of the lexicographic order on $[n]^{(k)}$.
\end{definition}

\begin{definition}
If $\mathcal{F} \subset [n]^{(k)}$ and $s \in \mathbb{N}$, the {\em $s$-upper shadow} of $\F$ is defined by
$$\partial^{+(s)}(\F) = \{B \in [n]^{(k+s)}:\ A \supset B \textrm{ for some }A \in \F\}.$$
\end{definition}

We need the following.
\begin{lemma}
\label{lemma:KK}
Let $\mathcal{F} \subset \p([n])$ be an increasing family, let $0 < p < p_0 < 1$, and let $x >0$. If $\mu_{p_0}(\mathcal{F}) \leq p_0^t(1-(1-p_0)^x)$, then $\mu_{p}(\mathcal{F}) \leq p^t(1-(1-p)^x)$.
\end{lemma}
Our proof employs the Dinur-Safra / Frankl-Tokushige method of `going to infinity and back', together with the Kruskal-Katona theorem.
\begin{proof}[Proof of Lemma \ref{lemma:KK}]
 Let $\mathcal{F} \subset \p([n])$ be an increasing family such that $\mu_{p_0}(\mathcal{F}) \leq p_0^t(1-(1-p_0)^x)$. Let $\epsilon >0$. Choose $n_0 = n_0(\epsilon) \in \mathbb{N}$ and a lexicographically ordered family $\mathcal{L} = \mathcal{L}(\epsilon) \subset \p([n+n_0])$ such that
$$\mu_{p_0}(\F) < \mu_{p_0}(\mathcal{L}) < \mu_{p_0}(\F) + \epsilon.$$
Note that $\mathcal{L}$ is increasing.

Recall from the proof of Theorem \ref{Thm:Katona-Our} that we define
$$\mathcal{F}_{N,k} : = \{A \in [N]^{(k)}:\ A \cap [n] \in \mathcal{F}\},$$
and that for any $\mathcal{F} \subset \mathcal{P}([n])$ and any $p \in (0,1)$, we have
\begin{equation}\label{eq:limit} \mu_p(\mathcal{F}) = \lim_{N \to \infty} \frac{|\mathcal{F}_{N, \lfloor pN \rfloor}|}{{N \choose \lfloor pN \rfloor}}.\end{equation}
It follows that if $N$ is sufficiently large depending on $n$, $n_0$ and $p_0$, we have
$$|\mathcal{F}_{N, \lfloor p_0 N \rfloor}| < |\mathcal{L}_{N, \lfloor p_0 N \rfloor}|.$$
Note that for any $N \geq k_0 \geq k \geq n$, and any increasing family $\mathcal{G} \subset \p([n])$ we have
$$\partial^{+(k_0-k)}(\mathcal{G}_{N,k}) = \mathcal{G}_{N,k_0}.$$
Moreover, for any $k \geq n$, $\mathcal{L}_{N,k}$ is lexicographically ordered (as a subset of $[N]^{(k)}$). Now let $k_0 := \lfloor p_0 N \rfloor$, and let $n \leq k \leq k_0$. We claim that $|\mathcal{F}_{N,k}| < |\mathcal{L}_{N,k}|$. Indeed, suppose for a contradiction that $|\mathcal{F}_{N,k}| \geq |\mathcal{L}_{N,k}|$. Then, using the Kruskal-Katona theorem, we have
$$|\mathcal{F}_{N,k_0}| = |\partial^{+(k_0-k)}(\mathcal{F}_{N,k})| \geq |\partial^{+(k_0-k)}(\mathcal{L}_{N,k})| = |\mathcal{L}_{N,k_0}| >  |\mathcal{F}_{N,k_0}|,$$
a contradiction. It follows from this claim, and from (\ref{eq:limit}), that
\begin{equation}
\label{eq:upper-bound-p} \mu_p(\F) \leq \mu_p(\mathcal{L}).
\end{equation}

Provided $\epsilon$ is sufficiently small depending on $p_0,t$ and $x$, we have
$$p_0^t(1-(1-p_0)^x) + \epsilon < p_0^t,$$
and therefore $\mu_{p_0}(\mathcal{L}) < p_0^t$. Hence, we may write $\mu_{p_0}(\mathcal{L}) = p_0^t(1-(1-p_0)^y)$, where $y >0$. Since $\mathcal{L}$ is lexicographically ordered, it is contained in the $t$-umvirate $\mathcal{S}_{[t]}$, and we have $\mu_{p_0}(\mathcal{L}^{[t]}_{[t]}) = 1-(1-p_0)^y$. Hence, by Lemma \ref{lemma:measures-decrease}, and since $\mathcal{L}^{[t]}_{[t]}$ is increasing, we have $\mu_{p}(\mathcal{L}^{[t]}_{[t]}) \leq 1-(1-p)^y$. Therefore, $\mu_{p}(\mathcal{L}) \leq p^t(1-(1-p)^y)$. Combining this with (\ref{eq:upper-bound-p}) yields
$$\mu_p(\F) \leq p^t(1-(1-p)^y).$$
As $\epsilon \to 0$, we must have $y \to x$, so taking the limit of the above as $\epsilon \to 0$ yields
$$\mu_p(\F) \leq p^t(1-(1-p)^x).$$
\end{proof}

We also need the following immediate corollary of Proposition \ref{lem:Weak Stability-general}.

\begin{corollary}
\label{corr:large-p-small-p}
Let $0 < p_1 <p_{0} <1$ and let $t \in \mathbb{N}$. For each $\delta >0$, there exists $x = x(\delta,p_{0},p_1,t)>0$ such that the following holds. If $\mathcal{F} \subset \p([n])$ is an increasing family with $\mu_{p_0}(\mathcal{F}) \leq p_{0}^{t}$, and
$$\mu_{p_1}(\mathcal{F}) \geq p_1^{t}(1-(1-p_1)^{x}),$$
then there exists $B \in [n]^{(t)}$ such that
$\mu_{p_1}(\mathcal{F} \setminus \mathcal{S}_{B}) \leq \delta p_1^{t-1}(1-p_1)$.
\end{corollary}

\begin{proof}[Proof of Proposition \ref{prop:small-p-rough}.]
Let $p_0,p,t,\delta,\mathcal{F}$ be as in the statement of the proposition. Let $C>0$ to be specified later. Since $\mu_{p_0}(\mathcal{F}) \leq p_0^t$, we have $\mu_p(\mathcal{F}) \leq p^t$, by Lemma \ref{lemma:measures-decrease}. Hence, by choosing $C > 2/p_0$, we may assume that $p < p_0/2$. If $Cp \geq 1$, then $\mu_p(\mathcal{F}) \geq p^t$, so $\mu_{p_0}(\mathcal{F}) = p_0^t$, $\mu_p(\mathcal{F}) = p^t$ and $\mathcal{F}$ is a $t$-umvirate, by Lemma \ref{lemma:measures-decrease}. Hence, we may assume that $Cp< 1$. Choose $y >0$ such that $Cp = 1-(1-p)^{y}$; then
$$\mu_p(\mathcal{F}) \geq Cp^{t+1} = p^{t}\left(1-\left(1-p\right)^{y}\right).$$
Set $p_1= p_{0}/2$. Then $\mu_{p_1}(\mathcal{F})\ge p_1^{t}\left(1-\left(1-p_1\right)^{y}\right)$,
by Lemma \ref{lemma:KK}. We have
$$y = \log_{1-p}(1-Cp) = \frac{\ln(1-Cp)}{\ln(1-p)} \geq C/2 \quad \forall p \leq 1/2,$$
so provided $C = C(\delta,p_0,t)$ is sufficiently large, we have $y>x(\delta,p_{0},p_0/2,t)$ (where $x(\delta,p_0,p_1,t)$ is as in Corollary \ref{corr:large-p-small-p}), so there exists $B \in [n]^{(t)}$ such that
$$\mu_{p_1}(\mathcal{F} \setminus \mathcal{S}_{B}) \leq \delta p_1^{t-1}(1-p_1).$$
It follows from Lemma \ref{lem:bootstrapping-General} (a) that
$$\mu_{p}(\mathcal{F} \setminus \mathcal{S}_{B}) \leq \delta^{\log_{p_1}(p)} p^{t-1}(1-p),$$
as required.
\end{proof}

\subsection{A bootstrapping argument}
\label{sec:sub:bootstrapping}

In this subsection we present a bootstrapping argument showing that if an increasing family $\F$ satisfies $\mu_{p_0}(\F) \leq p_0^t$
and is somewhat close to a $t$-umvirate, then it must be very close to that $t$-umvirate.

\begin{proposition}\label{prop:bs-gen}
Let $t \in \mathbb{N}$, and let $0 < p < p_0 < 1$. Define $\tilde{c}:=\left(\frac{1-p_{0}}{p_{0}}\right)^{\log_{1-p_{0}}\left(1-p\right)}$ and $u:=\log_{p}(p_{0}) \cdot \log_{1-p_{0}}\left(1-p\right)$. Let $\mathcal{F} \subset \p([n])$ be an increasing family with $\mu_{p_{0}}\left(\mathcal{F}\right)\le p_{0}^t$ and with
$$\mu_{p}\left(\mathcal{F} \setminus \mathcal{S}_B\right)\le (1-p)p^{t}u^{\frac{1}{1-u}}$$
for some $B \in [n]^{(t)}$. Let $\epsilon >0$. If
\begin{equation}\label{Eq:Bootstrapping0}
\mu_{p}\left(\mathcal{F}\right)\ge p^t\left(1-\tilde{c}\epsilon^{\log_{p}(p_{0}) \cdot\log_{1-p_{0}}(1-p)}\right)+\left(1-p\right)p^{t-1}\epsilon,
\end{equation}
then $\mu_{p}\left(\mathcal{F} \setminus \mathcal{S}_B \right)\le \epsilon p^{t-1}(1-p)$.
\end{proposition}

The proof of Proposition \ref{prop:bs-gen} has two parts. First, we use an isoperimetric technique encapsulated in Lemma~\ref{lem:bootstrapping-General}(b) to obtain an upper bound on $\mu_p(\F \cap \s_B)$ (for some $t$-umvirate $\s_B$) in terms of $\mu_p(\F \setminus \s_B)$. (Clearly, the larger $\mu_p(\F \setminus \s_B)$ is, the smaller $\mu_p(\F \cap \s_B)$ can be, but Lemma~\ref{lem:bootstrapping-General}(b) yields a sharp bound.) This upper bound yields an upper bound on $\mu_p(\F)=\mu_p(\F \cap \s_B)+\mu_p(\F \setminus \s_B)$ in terms of $\mu_p(\F \setminus \s_B)$. Second, we study the latter bound and show that once $\mu_p(\F \setminus \s_B)$ is assured to be below a certain fixed value, increasing $\mu_p(\F \setminus \s_B)$ only decreases the maximal possible total measure $\mu_p(\F)$.

\begin{proof}[Proof of Proposition~\ref{prop:bs-gen}]
We may assume w.l.o.g.\ that $B=[t]$. Let $\delta = \frac{\mu_{p}(\mathcal{F} \setminus \mathcal{S}_B)}{(1-p)p^{t-1}}$. By Lemma~\ref{lem:bootstrapping-General}(b), we have
\begin{align}\label{Eq:Bootstrappingbasic-General}
\mu_p(\F) = \mu_p(\F \setminus \mathcal{S}_B)+\mu_p(\F \cap \mathcal{S}_B) \leq \left(1-p\right)p^{t-1}\delta+p^t\left(1-\tilde{c}\delta^{u}\right).
\end{align}
Consider the function $f\left(x\right)=\left(1-p\right)p^{t-1}x+p^t\left(1-\tilde{c}x^{u}\right)$. It is easy to see that $f$ attains its minimum when $1-p=\tilde{c}pux^{u-1}$, i.e. when
\begin{align*} x & = \left(\frac{\tilde{c}pu}{1-p}\right)^{\frac{1}{1-u}}
=u^{\frac{1}{1-u}}\left(\left(\frac{1-p_{0}}{p_{0}}\right)^{\log_{1-p_{0}}\left(1-p\right)}\frac{p}{1-p}\right)^
{\frac{1}{1-\log_{p}p_{0}\log_{1-p_{0}}\left(1-p\right)}}\\
 & =u^{\frac{1}{1-u}}p_{0}^{\frac{\log_{p_0}p-\log_{1-p_{0}}\left(1-p\right)}{1-\log_{p}p_{0}\log_{1-p_{0}}\left(1-p\right)}}
 = u^{\frac{1}{1-u}}p_{0}^{\log_{p_0}p} = u^{\frac{1}{1-u}} p,\end{align*}
 and that $f$ is strictly decreasing in the interval $[0,u^{\frac{1}{1-u}} p]$, which contains the point $\delta$. By (\ref{Eq:Bootstrapping0}) we have $\mu_p(\F) \geq f(\epsilon)$, while by~\eqref{Eq:Bootstrappingbasic-General} we have $f\left(\delta\right)\ge\mu_{p}\left(\F\right)$. Hence, $f\left(\delta\right)\ge f\left(\epsilon\right)$, and consequently, $\delta \leq \epsilon$, or equivalently, $\mu_{p}\left(\mathcal{F} \setminus \mathcal{S}_B \right)\le \epsilon p^{t-1}(1-p)$, as asserted.
\end{proof}

In order to complete the bootstrapping argument in the case of large $p$, we will need the following technical claim.
\begin{claim}\label{Claim:continuity}
In the notation of Proposition~\ref{prop:bs-gen}, let
$$\nu(p_0,\zeta):= \inf_{p \in [\zeta p_0,p_0)}u^{\frac{1}{1-u}}.$$
Then for any $\eta>0$, there exists $B=B(\zeta,\eta)>0$ such that for all $p_0\leq 1-\eta$, we have
$$\nu(p_0,\zeta) \geq B.$$
\end{claim}

\begin{proof}
Since $u \mapsto u^{1/(1-u)}$ is an increasing function of $u$, and $u = u(p)$ is an increasing function of $p$ for fixed $p_0$,
it suffices to show that there exists $B=B(\zeta,\eta)>0$ such that if $p_0 \leq 1-\eta$, then
\[
u(\zeta p_0) = \log_{\zeta p_0}(p_0) \log_{1-p_0}(1-\zeta p_0) \geq B.
\]
This is easily verified.
\end{proof}

\subsubsection*{An enhanced bootstrapping result for $t$-intersecting families}

If instead of assuming $\mu_{p_0}(\F) \leq p_0^t$, we assume that $\F$ is $t$-intersecting, then a stronger bootstrapping result can be obtained by utilising Katona's shadow/intersection theorem (in the form of Theorem \ref{Thm:Katona-Our}). To show this, we need an analogue of Lemma~\ref{lem:bootstrapping-General}.
\begin{lemma}\label{lem:bootstrapping-intersecting}
Let $t \in \mathbb{N}$ and $0 < p \leq 1/2$. Suppose that $\mathcal{F}\subset \mathcal{P}\left(\left[n\right]\right)$ is an increasing $t$-intersecting
family with
\[
\mu_p(\F \setminus \s_{[t]}) \geq (1-p)p^{t-1}\epsilon.
\]
Then
\begin{equation}
\mu_{p}\left(\mathcal{F} \cap \s_{[t]} \right)\le p^t\left(1-\left(\frac{\epsilon}{2^t-1}\right)^{\log_{p}(1-p)}\right).
\end{equation}
\end{lemma}

\begin{proof}
It is sufficient to prove that for any $\delta>0$, if $\mu_{p}\left(\mathcal{F} \cap \s_{[t]}\right)\ge p^t(1-\delta)$,
then
\[
\mu_{p}\left(\mathcal{F}\setminus\mathcal{S}_{\left[t\right]}\right)\le\left(2^{t}-1\right)p^{t-1}\left(1-p\right)\delta^{\log_{1-p}p}.
\]
We show that for any $B\subsetneq\left[t\right]$, we have
\[
\mu_{p}\left(\mathcal{F}_{\left[t\right]}^{B}\right)p^{\left|B\right|}\left(1-p\right)^{t-\left|B\right|}\le p^{t-1}\left(1-p\right)\delta^{\log_{1-p}p},
\]
which clearly implies the assertion.

First, we note that the family $\mathcal{F}_{\left[t\right]}^{B} \subset \p(\{t+1,\ldots,n\})$ is $(t-\left|B\right|)$-intersecting (and in particular, $(t-\left|B\right|-1)$-intersecting). In addition, the families $\partial^{t-\left|B\right|-1}\left(\mathcal{F}_{\left[t\right]}^{B}\right)$
and $\mathcal{F}_{\left[t\right]}^{\left[t\right]}$ are cross-intersecting. Hence, we have
\[
\frac{\left(1-p\right)^{t-\left|B\right|-1}}{p^{t-\left|B\right|-1}}\mu_{p}\left(\mathcal{F}_{\left[t\right]}^{B}\right)
\le\mu_{p}\left(\partial^{t-\left|B\right|-1}\left(\mathcal{F}_{\left[t\right]}^{B}\right)\right)\le\delta^{\log_{1-p}p},
\]
where the first inequality uses Theorem~\ref{Thm:Katona-Our} and the second uses Lemma~\ref{Lemma:Cross-Intersecting}.
Rearrangement completes the proof.
\end{proof}
\noindent Now we are ready to state the enhanced bootstrapping result.
\begin{proposition}\label{prop:bs-int}
Let $t \in \mathbb{N}$ and let $0 <p<1/(t+1)$. Let $\mathcal{F} \subset \p([n])$ be an increasing, $t$-intersecting family with $\mu_{p}\left(\mathcal{F} \setminus \mathcal{S}_B\right)\le (1-p)p^{t}(c'v)^{\frac{1}{1-v}}$ for some $B \in [n]^{(t)}$,
where $c':=\left(2^t-1\right)^{-\log_{p}\left(1-p\right)}$ and $v:=\log_{p}\left(1-p\right)$. Let $\epsilon >0$. If
\begin{equation}\label{Eq:Bootstrapping-intersecting0}
\mu_{p}\left(\mathcal{F}\right)\ge p^t\left(1-\left(\frac{\epsilon}{2^t-1}\right)^{\log_{p}1-p}\right) + \left(1-p\right)p^{t-1}\epsilon,
\end{equation}
then $\mu_{p}\left(\mathcal{F} \setminus \mathcal{S}_B \right)\le \epsilon p^{t-1}(1-p)$.
\end{proposition}

\begin{proof}
The proof is essentially the same as the proof of Proposition~\ref{prop:bs-gen} above, with $c',v$ in place of $\tilde{c},u$, and Lemma~\ref{lem:bootstrapping-intersecting} in place of Lemma~\ref{lem:bootstrapping-General}(b).

Let $\delta = \frac{\mu_{p}(\mathcal{F} \setminus \mathcal{S}_B)}{(1-p)p^{t-1}}$. By Lemma~\ref{lem:bootstrapping-intersecting}, we have
\begin{align}\label{Eq:Bootstrapping1-General}
\mu_p(\F) = \mu_p(\F \setminus \mathcal{S}_B)+\mu_p(\F \cap \mathcal{S}_B) \leq \left(1-p\right)p^{t-1}\delta+p^t\left(1-c'\delta^{v}\right).
\end{align}
Consider the function $g\left(x\right)=\left(1-p\right)p^{t-1}x+p^t\left(1-c'x^{v}\right)$. Observe that $g$ attains its minimum when $1-p=c'pvx^{v-1}$, i.e. when
\[
x = \left(\frac{c'pv}{1-p}\right)^{\frac{1}{1-v}}=(c'v)^{\frac{1}{1-v}}\left(\frac{p}{1-p}\right)^{\frac{1}{1-v}}= (c'v)^{\frac{1}{1-v}}p,
\]
 and that $g$ is strictly decreasing in the interval $[0,(c'v)^{\frac{1}{1-v}} p]$, which contains the point $\delta$. By the assumption we have $\mu_p(\F) \geq g(\epsilon)$, while by~\eqref{Eq:Bootstrapping1-General} we have $g\left(\delta\right)\ge\mu_{p}\left(\F\right)$. Hence, $g\left(\delta\right)\ge g\left(\epsilon\right)$, and consequently, $\delta \leq \epsilon$, or equivalently, $\mu_{p}\left(\mathcal{F} \setminus \mathcal{S}_B \right)\le \epsilon p^{t-1}(1-p)$, as asserted.
\end{proof}

\subsection{Stronger stability results via bootstrapping}
\label{sec:sub:proof}
In this subsection, we complete the proof of Theorem \ref{Thm:Main}.

\subsubsection*{The case of large $p$}

In the case of large $p$ we prove the following result, which is slightly stronger than the assertion of Theorem~\ref{Thm:Main} in the relevant range.
\begin{proposition}\label{prop:large-p-gen}
For any $\zeta,\eta \in (0,1)$, there exists $c=c(\zeta,\eta)>0$ such that the following holds. Let $t \in \mathbb{N}$, and let $0 < \zeta p_{0} \leq p < p_{0} \leq 1-\eta$. Let $\mathcal{F} \subset \p([n])$ be an increasing family with $\mu_{p_{0}}\left(\mathcal{F}\right)\le p_{0}^t$ and
\begin{equation}
\label{eq:large-p-cond}
\mu_{p}\left(\mathcal{F}\right)\ge p^t\left(1-\frac{cp_0(p_0-p)}{\ln(\frac{1}{p_0})t}\right).
\end{equation}
Let $\epsilon >0$, and define $\tilde{c}:=\left(\frac{1-p_{0}}{p_{0}}\right)^{\log_{1-p_{0}}\left(1-p\right)}$. If
$$\mu_p(\F) \geq p^t\left(1-\tilde{c}\epsilon^{\log_{p}p_{0}\cdot\log_{1-p_{0}}1-p}\right)+\left(1-p\right)p^{t-1}\epsilon,$$
then there exists a $t$-umvirate $\mathcal{S}_B$ such that
$$\mu_{p}\left(\mathcal{F} \setminus \mathcal{S}_B \right) \le (1-p)p^{t-1}\epsilon.$$
\end{proposition}

\begin{proof}
Denote by $m$ the infimum of the function $p \mapsto pu^{\frac{1}{1-u}}$ in the interval
$[\zeta p_0,p_{0})$. By Claim~\ref{Claim:continuity}, since $p_0 \leq 1-\eta$, we have $m \geq K p_0$, where $K=K(\zeta,\eta)>0$ depends only on $\zeta,\eta$. By reducing $K$ if necessary, we may assume that $K \leq 1$. Let $\F$ be a family that satisfies the assumption of Proposition~\ref{prop:large-p-gen}, with $c=c(\zeta,\eta)$ to be specified later.
By Proposition~\ref{prop:bs-gen}, in order to show that $\F$ satisfies the conclusion of Proposition~\ref{prop:large-p-gen}, it suffices to show that there exists a $t$-umvirate $\mathcal{S}_B$ such that
\[
\mu_{p}\left(\mathcal{F} \setminus \mathcal{S}_B\right)\le (1-p)p^{t-1} Kp_0.
\]
And indeed, if $\mu_p(\F) \geq p^t(1-\frac{cp_0(p_0-p)}{\ln(\frac{1}{p_0})t})$, where $c=c(\zeta,\eta)>0$ is sufficiently small, then
$\mu_p(\F) \geq p^t(1-\epsilon')$, where $\epsilon' = \frac{C(\eta)Kp_0(p_0-p)}{\ln(\frac{1}{\zeta p_0})t}$, $C(\eta)$ being the constant of
Proposition~\ref{lem:Weak Stability-general}. Hence, by Proposition~\ref{lem:Weak Stability-general} (applied with $\epsilon = Kp_0$), we have $\mu_p\left(\mathcal{F} \setminus \mathcal{S}_B\right)\le (1-p)p^{t-1} K p_0$. This completes the proof.
\end{proof}

\subsubsection*{The case of small $p$}
\label{sec:sub:t>1,small-p}

For the case of small $p$, we prove the following.
\begin{proposition}\label{prop:small-p-gen}
For any $p_{0} \in\left(0,1\right)$ and $t \in \mathbb{N}$, there exists $C=C(p_0,t)>0$ such that the following holds. Let $0 < p < p_{0}$, and let $\mathcal{F} \subset \mathcal{P}([n])$ be an increasing family with $\mu_{p_{0}}\left(\mathcal{F}\right)\le p_{0}^t$ and $\mu_{p}\left(\mathcal{F}\right)\ge Cp^{t+1}$. Let $\epsilon >0$, and define $\tilde{c}:=\left(\frac{1-p_{0}}{p_{0}}\right)^{\log_{1-p_{0}}\left(1-p\right)}$. If
\[
\mu_{p}\left(\mathcal{F}\right)\ge p^t\left(1-\tilde{c}\epsilon^{\log_{p}p_{0}\cdot\log_{1-p_{0}}1-p}\right)+p^{t-1}\left(1-p\right)\epsilon,
\]
then there exists a $t$-umvirate $\mathcal{S}_B$ such that $\mu_{p}\left(\mathcal{F} \setminus \mathcal{S}_B \right)\le \epsilon p^{t-1}(1-p)$.
\end{proposition}

\begin{proof}
Let $\F \subset \mathcal{P}([n])$ be a family that satisfies the assumption of the proposition. By Lemma \ref{lemma:measures-decrease} (1), we have $\mu_p(\mathcal{F}) \leq p^t$. Let $\zeta = \zeta(p_0,t)>0$ to be chosen later. If we choose $C> 1/\zeta$, then $Cp^{t+1} > p^t$ for any $p \geq \zeta$, so we may assume throughout that $p < \zeta$.

We claim that the assertion of the proposition now follows from Propositions \ref{prop:small-p-rough} and \ref{prop:bs-gen}. Let $\delta = \delta(p_0,t)>0$ to be chosen later. By Proposition \ref{prop:small-p-rough}, provided $C=C(p_0,t)$ is sufficiently large, there exists $B \in [n]^{(t)}$ such that
$$\mu_{p}(\mathcal{F} \setminus \mathcal{S}_{B}) \leq \delta^{\log_{(p_0/2)}(p)} p^{t-1}(1-p).$$
Proposition \ref{prop:small-p-gen} will follow from Proposition \ref{prop:bs-gen} once we have shown that
$$\delta^{\log_{(p_0/2)}(p)} \leq pu^{\frac{1}{1-u}}.$$
We have
\[
pu^{\frac{1}{1-u}}=\Theta_{p_0} \left(\left(p \log_{p}p_{0}\log_{1-p_{0}}\left(1-p\right)\right)^{1+o\left(1\right)} \right) =\Theta_{p_0}(p^{2+o\left(1\right)}),
\]
(where $o(1)$ denotes a function of $p,p_0$ which tends to zero as $p \to 0$ for any fixed $p_0 \in (0,1)$). Hence, it suffices to prove that
$$\delta^{\log_{(p_0/2)}(p)} \leq \eta_{p_0} p^{2+o\left(1\right)}$$
(where $\eta_{p_0}>0$ depends only on $p_0$), i.e., that
$$\frac{\ln(1/p) \ln(1/\delta)}{\ln(2/p_0)} \geq \ln(1/\eta_{p_0}) + (2+o(1))\ln(1/p).$$
This holds provided we choose $\delta < (p_0/2)^3$ and provided we choose $\zeta$ to be sufficiently small depending on $p_0$ and $t$ (recall that we are assuming that $p < \zeta$). Therefore, Proposition~\ref{prop:bs-gen} can be applied to $\mathcal{F}$ to yield the assertion of Proposition~\ref{prop:small-p-gen}.
\end{proof}

\subsubsection*{Wrapping up the proof of Theorem \ref{Thm:Main}}

Theorem~\ref{Thm:Main} follows quickly from Propositions~\ref{prop:large-p-gen} and~\ref{prop:small-p-gen}.
Indeed, let $C=C(p_0,t)$ be the constant in Proposition~\ref{prop:small-p-gen}. By increasing the value of $C$ if necessary, we may assume that $C>2/p_0$. Define $\zeta_0 = 1/(Cp_0)$; note that $\zeta_0 <1/2$. Now apply Proposition~\ref{prop:large-p-gen} with $\eta = 1-p_0$ and $\zeta = \zeta_0$, for all $p \in [\zeta_0 p_0,p_0)$, yielding $c = c(p_0,t)>0$ such that for all $p \in [\zeta_0 p_0,p_0)$, the conclusion of Proposition \ref{prop:large-p-gen} holds when the condition (\ref{eq:large-p-cond}) is replaced by the condition
$$\mu_p(\F) \geq p^t(1-c(p_0-p)).$$
Apply Proposition~\ref{prop:small-p-gen} for all $p \in (0,\zeta_0 p_0)$. This completes the proof of Theorem~\ref{Thm:Main}.

\subsection{$t$-Intersecting families -- a stability result for the biased Wilson theorem}
\label{sec:sub:t-intersecting}

By the biased Wilson theorem (Theorem \ref{thm:Wilson-biased}), if $\F \subset \p([n])$ is a $t$-intersecting family, then $\mu_{1/(t+1)}(\F) \leq (t+1)^{-t}$. Hence, a direct application of Theorem~\ref{Thm:Main} with $p_0=1/(t+1)$ already yields a rather strong stability version of the biased Wilson theorem. However, for $t > 1$, the $\epsilon$-dependence in Theorem \ref{Thm:Main} is not sharp for $t$-intersecting families. (The tightness example $\tilde{\h}_{t,s,r}$ for Theorem \ref{Thm:Main} is not $t$-intersecting for $t \geq 1$ and $r > s$, and when $t >1$, the condition $(1-p_0)^{r-1} = p_0^{s-1}$ implies $r > s$.) By utilising the $t$-intersection condition, we obtain the following stability version of the biased Wilson theorem.
\begin{theorem}\label{thm:t-intersecting2-intro}
For any $t \in \mathbb{N}$, there exist $C=C(t) > 2(t+1),\ c= c(t) >0$ such that the following holds. Let $0 <p<1/(t+1)$, and let $\mathcal{F} \subset \p([n])$ be a $t$-intersecting family such that
\begin{equation}
\label{eq:extra-condition}
\mu_p(\F) \geq \begin{cases} p^t(1-c(\tfrac{1}{t+1}-p)) & \text{ if } p \geq 1/C\\ Cp^{t+1} & \text{ if } p < 1/C.\end{cases}
\end{equation}
Let $\epsilon>0$. If
\begin{equation}\label{Eq:Int-Condition}
\mu_{p}\left(\mathcal{F}\right)\ge p^t\left(1-\left(\frac{\epsilon}{2^t-1}\right)^{\log_{p}(1-p)}\right)+ (1-p)p^{t-1}\epsilon,
\end{equation}
then $\mu_p(\F \setminus \s_B) \leq (1-p)p^{t-1}\epsilon$ for some $t$-umvirate $\s_B$.
\end{theorem}
The $\epsilon$-dependence in \noindent Theorem~\ref{thm:t-intersecting2-intro} is sharp, up to a factor depending only on $t$, as evidenced by the families $\{\tilde{\mathcal{F}}_{t,s}\}_{t,s \in \mathbb{N}}$, defined as follows.
\begin{align*}
\tilde{\mathcal{F}}_{t,s} & :=\left\{ A\subset \mathcal{P}\left(\left[n\right]\right)\,:\,\left[t\right]\subset A,\,\left\{t+1,\ldots,t+s\right\}\cap A\ne\emptyset\right\} \\
 & \cup\left\{ A\subset \mathcal{P}\left(\left[n\right]\right)\,:\,|\left[t\right]\cap A|=t-1,\,\left\{t+1,\ldots,t+s\right\}\subset A\right\}.
\end{align*}
Indeed, we have
\[
\mu_{p}\left(\tilde{\mathcal{F}}_{t,s}\right)=p^{t}\left(1-\left(1-p\right)^{s}\right)+tp^{t-1}\left(1-p\right)p^{s}
\]
and
\[
\mu_{p}\left(\tilde{\mathcal{F}}_{t,s}\setminus\mathcal{S}_{\left[t\right]}\right)=tp^{t-1}\left(1-p\right)p^{s},
\]
which corresponds to the assertion of Theorem~\ref{thm:t-intersecting2-intro} with $\epsilon=tp^{s}$, provided we replace $2^{t}-1$ by $t$ in the condition (\ref{Eq:Int-Condition}). (Note that the condition (\ref{eq:extra-condition}) is satisfied provided $s$ is sufficiently large.) In fact, we
conjecture that the families $\tilde{\mathcal{F}}_{t,s}$ are precisely extremal (that is, the factor of $2^t-1$ in the statement of the theorem could be replaced by $t$); see Section \ref{sec:open-problems}.

The proof of Theorem \ref{thm:t-intersecting2-intro} is almost exactly the same as the proof of Theorem \ref{Thm:Main}, except that the enhanced bootstrapping result for $t$-intersecting families (Proposition \ref{prop:bs-int}) is used in place of the bootstrapping result for arbitrary increasing families (Proposition \ref{prop:bs-gen}); we omit the details.

\subsection{A stronger stability result for the Simonovits-S\'{o}s conjecture.}

Let $\mathcal{F}$ be a family of labelled graphs with vertex-set $[n]$, i.e. $\mathcal{F} \subset \mathcal{P}([n]^{(2)})$. The family $\mathcal{F}$ is said to be {\em triangle-intersecting} if any two graphs in $\mathcal{F}$ share some triangle. A well-known conjecture of Simonovits and S\'{o}s from 1976 asserted that if $\F \subset \mathcal{P}([n]^{(2)})$ is triangle-intersecting, then $|\F| \leq \frac{1}{8}2^{{{n}\choose{2}}}$. In 2012, Ellis, Filmus and Friedgut~\cite{EFF12} proved this conjecture, and in fact proved the slightly stronger statement that any triangle-intersecting family $\F \subset \mathcal{P}([n]^{(2)})$ satisfies
$\mu_p(\F) \leq p^3$ for all $p \leq 1/2$. Furthermore, they proved in \cite{EFF12} a stability version stating that if a triangle-intersecting $\F \subset \p([n]^{(2)})$
satisfies $\mu_p(\F) \geq (1-\epsilon) p^3$, then there exists a triangle $T$ such that $\mu_p(\F \setminus \s_T) \leq c \epsilon$, where
$\s_T := \{G \subset [n]^{(2)}:\ T \subset G\}$ and $c$ is an absolute constant.

Theorem~\ref{Thm:Main} (applied for $t=3$ and $p_0=1/2$), along with the result of~\cite{EFF12} that any triangle-intersecting family $\F$ satisfies $\mu_{1/2}(\F) \leq 1/8$, yields the following stability result.
\begin{corollary}\label{Thm:New-Simonovits-Sos}
There exist absolute constants $c>0,C>4$ such that the following holds. Let $0< p<1/2$ and let $\mathcal{F} \subset \mathcal{P}([n]^{(2)})$ be a
triangle-intersecting family with
$$\mu_p(\F) \geq \begin{cases} p^3(1-c(\tfrac{1}{2}-p)) & \text{ if } p \geq 1/C\\ Cp^4 & \text{ if } p < 1/C.\end{cases}$$
Let $\epsilon>0$. If
\begin{equation*}
\mu_{p}\left(\mathcal{F}\right)\ge p^3\left(1-\epsilon^{\log_{p}(1-p)}\right)+p^2\left(1-p\right)\epsilon,
\end{equation*}
then $\mu_p(\F \setminus \s_T) \leq (1-p)p^2\epsilon$ for some triangle $T$.
\end{corollary}
Corollary~\ref{Thm:New-Simonovits-Sos} is much stronger than the stability result in \cite{EFF12} in the case where $p$ is bounded away from $1/2$.
For example, if $\mu_p(\mathcal{F}) = (1-\delta)p^3$ and $p=1/4$, then it yields $\mu_{1/4}(\F \setminus \s_T) = O(\delta^{\log_{p}(1-p)})= O(\delta^{4.8})$, compared to $\mu_{1/4}(\F \setminus \s_T) = O(\delta)$ from the stability result of \cite{EFF12}. We believe, however, that
Corollary~\ref{Thm:New-Simonovits-Sos} is not tight in its $\epsilon$-dependence, and that to obtain a tight result one would have to exploit in a more significant way the fact that the family is triangle-intersecting (not just $3$-intersecting); see Section \ref{sec:open-problems}.

\subsection{The dual stability theorem and an application to the Erd\H{o}s matching conjecture}
\label{sec:dual}

We now give our dual version of Theorem~\ref{Thm:Main}, which allows us to obtain stability results for EKR-type theorems in which the extremal example is the {\em dual} of a $t$-umvirate.

\begin{theorem}\label{Thm:Dual}
Let $s\in\mathbb{N}$, and let $0 < p_{0} < 1$. Then there exist $C=C\left(p_{0},s\right) > \max\{2/p_0,s^2\}$ and $c=c\left(p_{0},s\right)>0$ such that
the following holds. Let $0 < p < p_0$, and let $\mathcal{F}\subset \mathcal{P}\left(\left[n\right]\right)$ be an increasing family such that
$\mu_{p_{0}}\left(\mathcal{F}\right)\le1-\left(1-p_{0}\right)^{s}$ and
\begin{equation}\label{eq:condition0-dual}
\mu_{p}\left(\mathcal{F}\right)\ge \begin{cases} (1-c(p_0-p))(1-(1-p)^s) & \text{ if } p \geq 1/C\\ (s-1)p+\tfrac{1}{2}Cp^{2} & \text{ if } p < 1/C.\end{cases}
\end{equation}
Let $\epsilon >0$, and define $\tilde{c}:=\left(\frac{1-p_{0}}{p_{0}}\right)^{\log_{1-p_{0}}\left(1-p\right)}$. If
\begin{equation}\label{eq:condition-dual}
\mu_{p}\left(\mathcal{F}\right)\ge1-\left(1-p\right)^{s-1} +(1-p)^{s-1} \left(p \left(1- \tilde{c}\epsilon^{\log_{p}p_{0}\log_{1-p_{0}} \left(1-p\right)}\right) + \left(1-p\right) \epsilon \right),
\end{equation}
then there exists $B \in [n]^{(s)}$ such that
\begin{equation}\label{eq:assertion-dual}
\mu_{p}\left(\mathcal{F}\setminus\mathrm{OR}_{B}\right)\le\left(1-p\right)^{s} \epsilon.
\end{equation}
\end{theorem}

\begin{remark}
For any $\xi\in (0,1/2]$, we have
$$\sup_{p_0 \in [\xi,1-\xi]} C(p_0,s) = O_{\xi,s}(1),\quad \inf_{p_0 \in [\xi,1-\xi]} c(p_0,s) = \Omega_{\xi,s}(1).$$
\end{remark}

\noindent Theorem~\ref{Thm:Dual} is tight for infinitely many of the families $\{\tilde{\mathcal{D}}_{s,d,l}:\ s,d,l \in \mathbb{N}\}$, defined by:
\begin{align*}
\tilde{\mathcal{D}}_{s,d,l} =& \{A \subset [n]: A \cap [s-1] \neq \emptyset\}\\
&\cup \{A \subset [n]: (A \cap [s]=\{s\}) \wedge (A \cap \{s+1,s+2,\ldots,s+d\}\neq \emptyset)\}\\
&\cup  \{A \subset [n]: (A \cap [s]=\emptyset) \wedge (\{s+1,s+2,\ldots,s+l\}\subset A)\}.
\end{align*}
To see this, for each $(s,d,l) \in \mathbb{N}^3$ with $d,l \geq 2$, choose the unique $p_0 \in (0,1)$ such that $(1-p_0)^{d-1}=p_0^{l-1}$. Then we have $\mu_{p_0}(\mathcal{D}_{s,d,l}) = 1-(1-p_0)^s$, $\mu_p(\mathcal{D}_{s,d,l}) =
1-(1-p)^{s-1} +(1-p)^{s-1}(p (1- (1-p)^d) + \left(1-p\right)p^l)$, and $\mu_p(\mathcal{D}_{s,d,l} \setminus \OR_{[s]}) = (1-p)^{s}p^{l}$, so equality holds in (\ref{eq:condition-dual}) and (\ref{eq:assertion-dual}) with $\epsilon=p^l$. Provided $d$ is sufficiently large, the condition (\ref{eq:condition0-dual}) is also satisfied.

\medskip \noindent While Theorem \ref{Thm:Dual} cannot be directly deduced from Theorem~\ref{Thm:Main}, the proof of Theorem \ref{Thm:Dual} is very similar indeed to that of Theorem~\ref{Thm:Main} (replacing $\mathcal{F}$ by its dual in the appropriate places), so is omitted.

\subsubsection*{Application to the Erd\H{o}s matching conjecture}

For $\mathcal{F} \subset \mathcal{P}([n])$, the \emph{matching number} $m(\F)$ of $\F$ is the maximum integer $s$ such that $\F$ contains $s$ pairwise disjoint sets. The well-known 1965 \emph{Erd\H{o}s matching conjecture}~\cite{Erdos65} asserts that if $n,k,s \in \mathbb{N}$ with $n \geq (s+1)k$ and $\F \subset [n]^{(k)}$ with $m(\F) \leq s$, then
$$|\F| \leq \max\left\{{n \choose k} - {n-s \choose k}, {k(s+1)-1 \choose k}\right\}.$$
This conjecture remains open. Erd\H{o}s himself proved the conjecture
for all $n$ sufficiently large, i.e. for $n \geq n_0(k,s)$. The bound on $n_0(k,s)$ was lowered in several works: Bollob\'as, Daykin and Erd\H{o}s~\cite{BDE76} showed
that $n_0(k, s) \leq 2sk^3$; Huang, Loh and Sudakov~\cite{HLS12} showed that $n_0(k,s) \leq 3sk^2$, and Frankl and F\"{u}redi (unpublished)
showed that $n_0(k, s) \leq cks^2$. The most significant result to date is the following theorem of Frankl~\cite{Frankl13}:
\begin{theorem}[Frankl, 2013]\label{Thm:Frankl-matching}
Let $n,k,s \in \mathbb{N}$ such that $n>(2s+1)k-s$, and let $\F \subset [n]^{(k)}$ such that $m(\F) \leq s$. Then
$|\F| \leq {{n}\choose{k}} - {{n-s}\choose{k}}$.
Equality holds if and only if there exists $B \in [n]^{(s)}$ such that $\F = \{A \in [n]^{(k)}:\ A \cap B \neq \emptyset\}$.
\end{theorem}
Frankl's Theorem immediately implies the following, via the method of `going to infinity and back'.
\begin{corollary}\label{Cor:EMCbiased}
Let $n,s \in \mathbb{N}$, and let $p\leq 1/(2s+1)$. Let $\F \subset \p([n])$ such that $m(\F)\leq s$. Then $\mu_p(\F) \leq 1-(1-p)^s$.
\end{corollary}

\noindent Using the $p=\tfrac{1}{2s+1}$ case of Corollary~\ref{Cor:EMCbiased}, we may apply Theorem~\ref{Thm:Dual} (with $p_0 = \tfrac{1}{2s+1}$) to immediately yield the following stability version of Corollary~\ref{Cor:EMCbiased}. (Note that in the proof, we may assume w.l.o.g. that $\mathcal{F}$ is increasing.)
\begin{corollary}\label{Cor:Matching}
For any $s \in \mathbb{N}$, there exist $C=C\left(s\right) > \max\{4s+2,s^2\},\ c=c\left(s\right)>0$ such that
the following holds. Let $0 < p < \tfrac{1}{2s+1}$, and let $\mathcal{F}\subset \mathcal{P}\left(\left[n\right]\right)$ such that $m(\F) \leq s$ and
\[
\mu_{p}\left(\mathcal{F}\right)\ge \begin{cases} (1-c(\tfrac{1}{2s+1}-p))(1-(1-p)^s) & \text{ if } p \geq 1/C\\ (s-1)p+\tfrac{1}{2}Cp^{2}& \text{ if }p < 1/C.\end{cases}\]
Let $\epsilon >0$, and define $\tilde{c}:=\left(2s\right)^{\log_{2s/(2s+1)}\left(1-p\right)}$. If
\[
\mu_{p}\left(\mathcal{F}\right)\ge 1-\left(1-p\right)^{s} - \left(1-p\right)^{s-1}p \tilde{c}
\epsilon^{\log_{p}(1/(2s+1))\log_{2s/(2s+1)}\left(1-p\right)}+ \left(1-p\right)^s\epsilon,
\]
then there exists $B \in [n]^{(s)}$ such that
\[
\mu_{p}\left(\mathcal{F}\setminus\mathrm{OR}_{B}\right)\le\left(1-p\right)^{s}\epsilon.
\]
\end{corollary}

\subsection{Additional applications}

It is easy to show that Theorem~\ref{Thm:Main} can be used to obtain a variety of other stability results for EKR-type theorems.
For example, it implies directly (a stronger version of) all results of~\cite{Kamat11} and the main result of~\cite{Tokushige11},
as well as stability versions of all currently known exact results on $r$-wise (cross)-$t$-intersecting families (see, e.g.,~\cite{FLST14,Tokushige10,Tokushige11} and the references therein). As these derivations are straightforward, we do not present
them in this paper.

\section{Families of $k$-element sets}
\label{sec:uniform}

In this section we leverage our main results from the biased-measure setting to the more classical setting of subfamilies of $[n]^{(k)}$, often called the `$k$-uniform' setting. Most of the section is devoted to the proof of
Theorem~\ref{thm:stability-uniform-t}, i.e., a stability result for Wilson's theorem. After that, we present the $k$-uniform versions of our stability results for the Simonovits-S\'{o}s conjecture and the Erd\H{o}s matching conjecture.

Throughout this section, if $\F \subset \p([n])$ and $k \in [n]$, we will often write $\F^{(k)}:=\F \cap [n]^{(k)}$ for brevity, abusing notation slightly.

We will make repeated use of the following simple Chernoff bound (see, e.g.,~\cite{AS}, Appendix~A).
\begin{proposition}
\label{prop:chernoff}
Let $X$ be a random variable with $X \sim \Bin(n,p)$, and let $\delta \in [0,1]$. Then
$$\Pr\{X \geq (1+\delta)np\} < e^{-\delta^2 np /3},\quad \Pr\{X \leq (1-\delta)np\} < e^{-\delta^2 np /2}.$$
\end{proposition}

\subsection{A stability result for Wilson's theorem}
\label{sec:sub:Wilson}

In this section we prove Theorem~\ref{thm:stability-uniform-t}, our almost-sharp stability result for Wilson's theorem (Theorem \ref{Thm:Wilson}), improving the stability result of Friedgut in~\cite{Friedgut08}. Let us recall the
statement of Theorem \ref{thm:stability-uniform-t}.
\begin{theorem*}
For any $t \in \mathbb{N}$ and $\eta >0$, there exists $\delta_{0}=\delta_0(\eta,t)>0$ such that the following holds. Let $k,n \in \mathbb{N}$ with $\tfrac{k}{n} \leq \tfrac{1}{t+1}-\eta$, and let $d\in\mathbb{N}$. Let $\mathcal{A}\subset [n]^{(k)}$ be a $t$-intersecting family with
\[
\left|\A\right| > \max\left\{ \binom{n-t}{k-t}\left(1-\delta_{0}\right),\binom{n-t}{k-t}-\binom{n-t-d}{k-t}+\left(2^{t}-1\right)\binom{n-t-d}{k-t-d+1}\right\}.
\]
Then there exists a $t$-umvirate $\s_B$ such that
\[
\left|\mathcal{A} \setminus \mathcal{S}_{B} \right| \leq \left(2^{t}-1\right)\binom{n-t-d}{k-t-d+1}.
\]
\end{theorem*}

As in the proof of Theorem~\ref{Thm:Main}, we first prove a weak stability result, and then we prove a bootstrapping lemma that allows us to leverage our weak stability result into a stronger stability result.

\subsubsection{A weak stability result}

We start with two lemmas. (Recall that we use the convention ${a \choose b} = 0$ if $a,b \in \mathbb{Z}$ with $a <0$ or $b <0$.)
\begin{lemma}
\label{lemma:cross-intersecting-slice}Let $n,k,l \in \mathbb{N}$ with $n \geq k+l$, let $r \in \mathbb{N} \cup \{0\}$, and let $\mathcal{A}\subset [n]^{(k)},\ \mathcal{B} \subset [n]^{(l)}$ be cross-intersecting families. Suppose that $\left|\mathcal{A}\right| \geq \binom{n}{k}-\binom{n-r}{k}$. Then $\left|\mathcal{B}\right|\le\binom{n-r}{l-r}$.
\end{lemma}

\begin{proof}
Let $\overline{\mathcal{A}}:=\left\{[n] \setminus A\,:\, A\in\mathcal{A}\right\}$; then $\left|\overline{\mathcal{A}}\right| \geq \binom{n}{n-k}-\binom{n-r}{n-k-r}$. By the Kruskal-Katona theorem, we have $\left|\partial^{n-k-l}\left(\overline{\mathcal{A}}\right)\right| \geq {n \choose l} - {n-r \choose l-r}$. Since $\mathcal{A}$ and $\mathcal{B}$ are cross-intersecting, we have $\mathcal{B}\cap\partial^{n-k-l}\left(\overline{\mathcal{A}}\right)=\emptyset$. Hence, $\left|\mathcal{B}\right|\le\binom{n}{l}-\left|\partial^{n-k-l}\left(\overline{\mathcal{A}}\right)\right|\le\binom{n-r}{l-r}$, as required.
\end{proof}

\noindent Combining this with Theorem~\ref{Thm:Katona} yields the following.
\begin{lemma}
\label{lemma:union-bound}
Let $n,k,l,t \in\mathbb{N}$ with $n \geq k+l-2t+1$, let $r \in \mathbb{N} \cup \{0\}$, and let $B \in [n]^{(t)}$. Let $\mathcal{F} \in \mathcal{P}([n])$ be a $t$-intersecting family such that
$$\left|\mathcal{F}^{(k)}\cap\mathcal{S}_{B}\right| \geq \binom{n-t}{k-t}-\binom{n-t-r}{k-t}.$$
Then $\left|\mathcal{F}^{(l)}\setminus\mathcal{S}_{B}\right| \leq \left(2^{t}-1\right)\binom{n-t-r}{l-t-r+1}$.
\end{lemma}
\begin{proof}
Without loss of generality, we may assume that $B=[t]$. It suffices to show that for each $C\subsetneq [t]$, we have $\left|(\mathcal{F}_{\left[t\right]}^{C})^{(l-|C|)}\right|\le\binom{n-t-r}{l-t-r+1}$.

Note that $\mathcal{F}_{\left[t\right]}^{C}$ is
$(t-\left|C\right|)$-intersecting (and in particular, $(t-\left|C\right|-1)$-intersecting),
and that $\partial^{t-\left|C\right|-1}\left(\mathcal{F}_{\left[t\right]}^{C}\right)$
and $\mathcal{F}_{\left[t\right]}^{\left[t\right]}$ are cross-intersecting. Applying Theorem~\ref{Thm:Katona} to $\F_{[t]}^C$ and
Lemma~\ref{lemma:cross-intersecting-slice} to the pair
$(\mathcal{F}_{\left[t\right]}^{\left[t\right]})^{(k-t)},\ \partial^{t-\left|C\right|-1}
\left(\left(\mathcal{F}_{\left[t\right]}^{C}\right)^{(l-|C|)}\right)$,
we get
\[
\left|(\mathcal{F}_{\left[t\right]}^{C})^{(l-|C|)}\right|\le\left|\partial^{t-\left|C\right|-1}
\left(\left(\mathcal{F}_{\left[t\right]}^{C}\right)^{(l-|C|)}\right)\right|\le\binom{n-t-r}{l-t-r+1},
\]
as asserted.
\end{proof}

This, together with our results for the $p$-biased measure on $\mathcal{P}([n])$, enables us to prove the following weak
stability result.
\begin{proposition}
\label{lemma:starting}
Let $t,n \in \mathbb{N}$, let $0 < \eta < \tfrac{1}{t+1}$, let $\epsilon >0$, and let $k\in\mathbb{N}$ be such that
$k\leq(\tfrac{1}{t+1}-\eta)n$. Then there exists $\delta = \delta(\epsilon,\eta,t)>0$ such that the following holds. Let $\A \subset [n]^{(k)}$ be a $t$-intersecting family with
$$|\mathcal{A}| \geq (1-\delta){n-t \choose k-t}.$$
Then there exists $B \in [n]^{(t)}$ such that
$$|\mathcal{A} \setminus \mathcal{S}_{B}| \leq \epsilon {n-t \choose k-t}.$$
\end{proposition}
\begin{proof}
By the equality case of Theorem \ref{Thm:Wilson}, by making $\delta = \delta(\epsilon,\eta,t)>0$ smaller if necessary,
we may assume throughout that $n \geq n_0(\epsilon,\eta,t)$ for any $n_0(\epsilon,\eta,t) \in \mathbb{N}$.

\mn Choose $\delta>0$ sufficiently small such that
$$|\mathcal{A}| \geq {n-t \choose k-t} - {n-t-s \choose k-t},$$
where $s = s(\epsilon,\eta,t) \in \mathbb{N}$ will be chosen later. Let $\mathcal{F} = \A^{\uparrow }$ be the minimal increasing subfamily of
$\p([n])$ that contains $\A$. Define
$$\mathcal{C}_{t,s}: = \{F \subset [n]:\ [t] \subset F,\ F \cap \{t+1,\ldots,t+s\} \neq \emptyset\};$$
note that $\mathcal{C}_{t,s}^{(l)}$ is an initial segment of the lexicographic ordering on $[n]^{(l)}$, for any $t+1 \leq l \leq n$. We have
$$|\F^{(k)}| = |\mathcal{A}| \geq {n-t \choose k-t} - {n-t-s \choose k-t} = |(\mathcal{C}_{t,s})^{(k)}|,$$
so by the Kruskal-Katona theorem, we have
\begin{equation}\label{eq:dom}|\mathcal{F}^{(l)}| \geq |(\mathcal{C}_{t,s})^{(l)}| = {n-t \choose l-t} - {n-t-s \choose l-t}\quad \forall \l \geq k.\end{equation}
Define $p_0 : = 1/(t+1)$, $p_1 : = k/n$, $p : = (p_0+p_1)/2$. By the Chernoff bound in Proposition \ref{prop:chernoff}, we have
\begin{equation}\label{eq:chernoff} \mu_p(\{F \subset [n]:|F| < k\}) \leq \Pr\{\Bin(n,p) < (1-\eta)pn\} < e^{-\eta^2 pn/2} \leq e^{-\eta^2n/(4(t+1))} = o_{\eta,t}(1),\end{equation}
where $o_{\eta,t}(1)$ denotes a function of $n$ tending to 0 as $n \to \infty$, for fixed $\eta,t$. Combining (\ref{eq:dom}) and (\ref{eq:chernoff}), we have
$$\mu_p(\mathcal{F}) \geq \mu_p(\mathcal{C}_{t,s}) - \mu_p(\{F \subset [n]:\ |F| < k\}) = p^t(1-(1-p)^s) - o_{\eta,t}(1) = (1-o_{\eta,t}(1))p^t(1-(1-p)^s).$$
On the other hand, since $\mathcal{F}$ is $t$-intersecting, we have $\mu_{p_0}(\F) \leq p_0^t$. Therefore, by Theorem \ref{thm:t-intersecting2-intro}, provided
\begin{equation} \label{eq:technical-condition} (1-p)^s + o_{\eta,t}(1) \leq c(\tfrac{1}{t+1}-p),\end{equation}
there exists $B \in [n]^{(t)}$ such that
$$\mu_{p}(\mathcal{F} \setminus \mathcal{S}_B) \leq (1-p)p^{t-1}\epsilon_1,$$
where $\epsilon_1$ is the smallest positive solution to
$$\left(\frac{\epsilon_1}{2^t-1}\right)^{\log_{p}(1-p)}- \frac{1-p}{p} \epsilon_1 = (1-p)^s + o_{\eta,t}(1).$$
Using the fact that $p \in [1/(2(t+1)),1/(t+1)-\eta/2]$, provided $s,n$ are sufficiently large depending on $\eta,t$, condition (\ref{eq:technical-condition}) holds and we have $\epsilon_1 \leq O_t(1) p^s + o_{\eta,t}(1)$.
Hence,
$$\mu_{p}(\mathcal{F} \setminus \mathcal{S}_B) \leq p^t(O_t(1) p^s + o_{\eta,t}(1)),$$
and therefore
\begin{equation}\label{eq:subcube-lower} \mu_p(\mathcal{F} \cap \mathcal{S}_B) \geq p^t(1-(1-p)^s - O_t(1) p^s - o_{\eta,t}(1)).\end{equation}
Now choose $r \in \mathbb{N}$ minimal such that
$$(2^t-1) \binom{n-t-r}{k-t-r+1} \leq \epsilon {n-t \choose k-t};$$
note that $r \leq r_0(\epsilon,t)$ for all $k \leq (1/(t+1)-\eta)n$. We claim that
\begin{equation}\label{Eq:UWS1}
|\mathcal{A} \setminus \mathcal{S}_{B}| \leq (2^t-1) \binom{n-t-r}{k-t-r+1},
\end{equation}
which implies the assertion of the proposition.

\mn Suppose on the contrary that~\eqref{Eq:UWS1} fails. Then by Lemma \ref{lemma:union-bound}, we have
\begin{equation}\label{eq:dom-2} |\mathcal{F}^{(l)} \cap \mathcal{S}_{B}| < \binom{n-t}{l-t}-\binom{n-t-r}{l-t} = |(\mathcal{C}_{t,r})^{(l)}| \quad \forall l \leq n-k+2t-1.\end{equation}
Using the Chernoff bound in Proposition \ref{prop:chernoff} as above, it is easy to see that
\begin{equation}\label{eq:chernoff-2}\mu_p(\{F \subset [n]:\ |F| > n-k+2t-1\}) = o_{\eta,t}(1).\end{equation}
Therefore, combining (\ref{eq:dom-2}) and (\ref{eq:chernoff-2}), we obtain
\begin{align*} \mu_p(\mathcal{F} \cap \mathcal{S}_B) & \leq \mu_p(\mathcal{C}_{t,r}) + \mu_p(\{F \subset [n]: |F| > n-k+2t-1\})\\
& = p^t(1-(1-p)^r) + o_{\eta,t}(1)\\
&= p^t((1-(1-p)^r + o_{\eta,t}(1)),
\end{align*}
contradicting (\ref{eq:subcube-lower}) if $s$ and $n$ are sufficiently large depending on $\eta$, $t$ and $\epsilon$. This completes the proof.
\end{proof}

\subsubsection{A bootstrapping argument}


First, we cite an old result of Hilton (see~\cite{Frankl87}, Theorem~1.2).

\begin{notation}
For $X \subset \mathbb{N}$, $i \in \mathbb{N}$ and $\mathcal{A}\subset X^{(i)}$, we write $\mathcal{L}(\mathcal{A})$ for the initial segment of the lexicographic order on $X^{(i)}$ with size $|\mathcal{A}|$. We say a family $\mathcal{C} \subset X^{(i)}$ is {\em lexicographically ordered} if it is an initial segment of the lexicographic order on $X^{(i)}$, i.e., $\mathcal{L}(\mathcal{C})=\mathcal{C}$.
\end{notation}

\begin{proposition}[Hilton]
\label{lem:cross-by-kk}
If $\A \subset [n]^{(k)}$, $\B \subset [n]^{(l)}$
are cross-intersecting, then $\mathcal{L}(\mathcal{A})$, $\mathcal{L}(\mathcal{B})$ are also cross-intersecting.
\end{proposition}

We use the following technical lemma.
\begin{lemma}
\label{lem:Technical}
For any $\eta >0$ and any $C \geq 0$, there exists $c = c_0(\eta,C)\in \mathbb{N}$ such that the following holds. Let $n,l,k,d \in \mathbb{N} \cup \{0\}$ with $n\ge (1+\eta)l+k+c$ and $l \geq k+c-1$. Suppose that $\mathcal{A} \subset [n]^{(l)},\, \B \subset [n]^{(k)}$
are cross-intersecting, and that
\[
\left|\A\right|\le\left|\mathrm{OR}_{\left[d\right]} \cap [n]^{(l)}\right|=\binom{n}{l}-\binom{n-d}{l}.
\]
Then
\[
\left|\mathcal{A}\right|+C\left|\mathcal{B}\right|\le\binom{n}{l}-\binom{n-d}{l}+C\binom{n-d}{k-d}.
\]
\end{lemma}
\begin{proof}
We prove the lemma by induction on $k$. For $k=0$ the lemma holds trivially. Assume now that $k \geq 1$, and that the statement of the lemma holds for $k-1$. For $d=0$, the statement of the lemma holds trivially, so we may assume throughout that $d \geq 1$. By Proposition~\ref{lem:cross-by-kk},
we may assume that $\mathcal{A}$ and $\mathcal{B}$ are lexicographically
ordered. Since $d \geq 1$, we have $|\mathcal{A}| \leq \binom{n}{l}-\binom{n-1}{l} = \binom{n-1}{l-1}$, so $\mathcal{A} \subset \F_1^{(l)}$, where $\F_{1}^{\left(i\right)}:=\left\{ A\in [n]^{(i)}:\,1\in A\right\}$ for each $i\in\left[n\right]$.

We split into two cases: $\A = \F_1^{(l)}$ and $\A \subsetneq \F_1^{(l)}$.

\mn \textbf{Case 1: $\A = \F^{(l)}_{1}$.} First note that $\mathcal{B}\subset \F_{1}^{\left(k\right)}$. Indeed,
suppose on the contrary that $B\in\B$ and $1\notin B$. Since $n \geq k+l$, there exists $A \in [n]^{(l)}$ such that $1 \in A$ and $A\cap B=\emptyset$.
Hence, $A\in\F_{1}^{\left(l\right)} = \A$, and $A\cap B=\emptyset$, a contradiction. Hence, we may assume that $\mathcal{B}= \F_{1}^{\left(k\right)}$. We must prove that
\begin{equation} \label{eq:basic-inequality} {n-1 \choose l-1} + C{n-1 \choose k-1} \leq \binom{n}{l}-\binom{n-d}{l}+C\binom{n-d}{k-d}\quad \forall d \geq 1.\end{equation}
This clearly holds (with equality) if $d=1$. To verify it for all $d \geq 2$ it suffices to show that
$${n-1 \choose l-1} + C{n-1 \choose k-1} \leq \binom{n}{l}-\binom{n-2}{l},$$
or equivalently,
$$C{n-1 \choose k-1} \leq \binom{n-2}{l-1}.$$
We have
\begin{align*} \frac{{n-1 \choose k-1}}{\binom{n-2}{l-1}} &= \frac{n-1}{n-k} \frac{{n-2 \choose k-1}}{\binom{n-2}{l-1}}
\leq 2\frac{(l-1)(l-2)\ldots k}{(n-k-1)(n-k-2)\ldots (n-l)}\\
& \leq 2 \left(\frac{l-1}{n-k-1}\right)^{l-k}
 \leq 2 \left(\frac{l-2}{l+\eta l + c -1}\right)^{c-1} \leq \frac{1}{C},
\end{align*}
provided $c$ is sufficiently large depending on $\eta$ and $C$, as required.

\mn \textbf{Case 2: $\A\subsetneq\F_1^{(l)}$.} If $\left|\A\right| \leq \binom{n-2}{l-2}$, then
\[
\left|\A\right|+C\left|\B\right|\le\binom{n-2}{l-2}+C\binom{n}{k} \leq \binom{n-1}{l-1} \leq \binom{n}{l}-\binom{n-d}{l}+C\binom{n-d}{k-d},
\]
where the second inequality holds since
\begin{align*}
\frac{{n \choose k}}{{n-1 \choose l-1}-{n-2 \choose l-2}} & = \frac{{n \choose k}}{{n-2 \choose l-1}}
=\frac{n(n-1)}{(n-k)(n-k-1)} \frac{{n-2 \choose k}}{{n-2 \choose l-1}}
\leq 4 \frac{(l-1)(l-2)\ldots (k+1)}{(n-k-2)(n-k-3)\ldots (n-l)}\\
& \leq 4 \left(\frac{l-1}{n-k-2}\right)^{l-k-1}
\leq 4\left(\frac{l-1}{l+\eta l + c -2}\right)^{c-2}
\leq \frac{1}{C},
\end{align*}
provided $c$ is sufficiently large depending on $\eta$ and $C$. Hence, we may assume that
$${n-2 \choose l-2} \leq |\mathcal{A}| \leq {n-1 \choose l-1}.$$
Therefore, since $\A$ is lexicographically ordered, we have $\mathcal{A} \supset \{S \in [n]^{(l)}:\ 1,2 \in S\}$. Hence, $B \cap \{1,2\} \neq \emptyset$ for all $B \in \mathcal{B}$. (If there exists $B \in \mathcal{B}$ with $B \cap \{1,2\} = \emptyset$, then since $n \geq k+l$, there exists $A \in [n]^{(l)}$ with $A \cap B = \emptyset$ and $1,2 \in A$, but the latter implies $A \in \mathcal{A}$, a contradiction.) Therefore, since $\B$ is lexicographically ordered, we have $\B \supset \F_1^{(k)}$.

Observe that
$$\mathcal{A}_{\left\{ 1,2\right\} }^{\left\{ 1\right\} }\subseteq ([n]\setminus [2])^{(l-1)},\quad \B_{\left\{ 1,2\right\} }^{\left\{ 2\right\} }\subset ([n]\setminus[2])^{(k-1)}$$
are cross-intersecting, and trivially $|\mathcal{A}_{\left\{ 1,2\right\} }^{\left\{ 1\right\} }| \leq {n-2 \choose l-1}$. Hence, by the induction hypothesis (which we may apply since $(n-2) \geq (1+\eta)(l-1) + (k-1) +c$ and $l-1 \geq k-1 + c-1$, choosing $d=n-2$), we have
\[
\left|\mathcal{A}_{\left\{ 1,2\right\} }^{\left\{ 1\right\} }\right|+C\left|\B_{\left\{ 1,2\right\} }^{\left\{ 2\right\} }\right| \leq {n-2 \choose l-1},\]
and therefore,
\begin{align*} |\A|+C|\B| & = {n-2 \choose l-2} + \left|\mathcal{A}_{\left\{ 1,2\right\} }^{\left\{ 1\right\} }\right| + C {n-1 \choose k-1} + C\left|\B_{\left\{ 1,2\right\} }^{\left\{ 2\right\} }\right| \\
&\leq {n-2 \choose l-2} + {n-2 \choose l-1} + C{n-1 \choose k-1}\\
&= {n-1 \choose l-1} + C{n-1 \choose k-1}\\
& \leq \binom{n}{l}-\binom{n-d}{l}+C\binom{n-d}{k-d},
\end{align*}
using (\ref{eq:basic-inequality}) for the last inequality. This completes the proof.
\end{proof}

We now give a corollary of Lemma~\ref{lem:Technical}, with a choice of parameters which will be useful later.
\begin{corollary}
\label{cor:cor}
For any $\eta >0$ and any $t \in \mathbb{N}$, there exists $c=c(\eta,t) \in \mathbb{N}$ such that the following holds. Let $k,n,d\in\mathbb{N}$ with $n>(2+\eta)k+c$ and $d \geq c$. Suppose
that
$$\C \subset ([n]\setminus [t])^{(k-t)},\quad \D \subset ([n] \setminus [t])^{(k+1-t)}$$
are cross-intersecting. Suppose also that
\[
\binom{n-t}{k-t}-\binom{n-t-c}{k-t} \leq \left|\C\right| \leq \binom{n-t}{k-t}-\binom{n-t-d}{k-t}.
\]
Then
\[
\left|\C\right|+\left(2^{t}-1\right)\left|\D\right|\le\binom{n-t}{k-t}-\binom{n-t-d}{k-t}+\left(2^{t}-1\right)\binom{n-t-d}{k-t-d+1}.
\]
\end{corollary}

\begin{proof}
By Lemma \ref{lem:cross-by-kk}, we may assume that $\C$ and $\D$ are lexicographically ordered, so that in particular $\mathcal{C} \supset \OR_{\{t+1,\ldots,t+c\}} \cap ([n] \setminus [t])^{(k-t)}$. Therefore,
$$\left|\C\right|=\binom{n-t}{k-t}-\binom{n-t-c}{k-t}+|\C_{\{t+1,\ldots,t+c\}}^{\emptyset}|.$$
Additionally (using $n>2k$) we have that all the sets in $\D$ contain
$\{t+1,\ldots,t+c\}$. Now note that
$$\C_{\{t+1,\ldots,t+c\}}^{\emptyset}\subset ([n] \setminus [t+c])^{(k-t)},\quad \D_{\left\{t+1,\ldots,t+c\right\}}^{\{t+1,\ldots,t+c\}}\subset ([n] \setminus [t+c])^{(k+1-t-c)}$$
are cross-intersecting. Using Lemma \ref{lem:Technical}, with $n' = n-t-c$, $k' = k-t-c+1$, $l' = k-t$ and $d' = d-c$, $C'=2^t-1$, where $c := c_0(\eta,C') = c_0(\eta,2^{t}-1)$, we obtain
\begin{align*}
\left|\C\right|+(2^{t}-1)\left|\D\right| & =\binom{n-t}{k-t}-\binom{n-t-c}{k-t}+\left|\C_{\left\{t+1,\ldots,c\right\}}^{\emptyset}\right|+\left(2^{t}-1\right)\left|\D_{\left\{t+1,\ldots,t+c\right\}}^{\left\{t+1,\ldots,t+c\right\}}\right|\\
 & \le\binom{n-t}{k-t}-\binom{n-t-c}{k-t} + \\
 & \qquad +\binom{n-t-c}{k-t} -\binom{n-t-d}{k-t}+\left(2^{t}-1\right)\binom{n-t-d}{k-t-d+1}\\
 & =\binom{n-t}{k-t}-\binom{n-t-d}{k-t}+\left(2^{t}-1\right)\binom{n-t-d}{k-t-d+1}.
\end{align*}
\end{proof}

\mn Now we are ready to prove our bootstrapping lemma.
\begin{proposition}
\label{lem:main}
For any $\eta >0$ and any $t \in \mathbb{N}$, there exists $c = c_1(\eta,t) \in \mathbb{N}$ such that the following holds. Let $k,n,d\in\mathbb{N}$ with $n>(2+\eta)k+c$ and $d \geq c$, let $\mathcal{A}\subset [n]^{(k)}$ be a $t$-intersecting family, and let $B\subset [n]^{(t)}$. Suppose that
\[
\binom{n-t}{k-t}-\binom{n-t-c}{k-t}\le\left|\mathcal{A}\cap\mathcal{S}_{B}\right|\le\binom{n-t}{k-t}-\binom{n-t-d}{k-t}.
\]
Then
$$\left|\mathcal{A}\right|\le\binom{n-t}{k-t}-\binom{n-t-d}{k-t}+\left(2^{t}-1\right)\binom{n-t-d}{k-t-d+1}.$$
 \end{proposition}

\begin{proof}
Denote $\mathcal{C} := \A_B^B$. Let $C_0 \subsetneq B$ be such that $|\partial^{t-\left|C_0\right|-1}(\mathcal{A}_{B}^{C_0})|$ is maximal amongst all
$C \subsetneq B$. Using Theorem~\ref{Thm:Katona} (which can be applied, since for any $C \subsetneq B$, $\A_B^{C}$ is $(t-|C|)$-intersecting),
we have
\begin{equation}\label{Eq:UWS2}
|\mathcal{A}| = |\mathcal{C}|+ \sum_{C \subsetneq B} |\mathcal{A}_{B}^{C}| \leq  |\mathcal{C}|+ \sum_{C \subsetneq B} |\partial^{t-\left|C\right|-1}(\mathcal{A}_{B}^{C})| \leq  |\mathcal{C}|+ (2^t-1) |\partial^{t-\left|C_0\right|-1}(\mathcal{A}_{B}^{C_0})|.
\end{equation}
As $\mathcal{C}$ and $\mathcal{D}=\partial^{t-\left|C_0\right|-1}(\mathcal{A}_{B}^{C_0})$ are cross-intersecting,
the assertion follows from~\eqref{Eq:UWS2} by applying Corollary~\ref{cor:cor} to $\mathcal{C}$ and $\mathcal{D}$.
\end{proof}

\subsubsection{Proof of Theorem \ref{thm:stability-uniform-t}}

Theorem \ref{thm:stability-uniform-t} follows easily from Propositions~\ref{lemma:starting} and~\ref{lem:main}.
\begin{proof}[Proof of Theorem \ref{thm:stability-uniform-t}.]
Let $n,k,t$ and $\eta$ be as in the statement of the theorem. Let $\delta_0 = \delta_0(\eta,t)>0$ to be chosen later. By the equality case of Theorem \ref{Thm:Wilson}, we may assume throughout that $n \geq n_0(\eta,t)$ for any $n_0(\eta,t) \in \mathbb{N}$, by choosing $\delta_0$ to be sufficiently small.

Let $\mathcal{A}\subseteq [n]^{(k)}$ be a $t$-intersecting family with
$$|\mathcal{A}| > \max\left\{(1-\delta_0){n-t \choose k-t}, \binom{n-t}{k-t}-\binom{n-t-d}{k-t}+\left(2^{t}-1\right)\binom{n-t-d}{k-t-d+1}\right\}.$$
Choose $\epsilon = \epsilon(\eta,t)>0$ such that
$$ {n-t \choose k-t}(1-2\epsilon) \geq {n-t \choose k-t} - \binom{n-t-c}{k-t},$$
where $c=c_1(\eta,t)$ is as given by Proposition \ref{lem:main}. Let $\delta = \delta(\epsilon,\eta,t)>0$ be as given by Proposition \ref{lemma:starting}. By reducing $\delta$ if necessary, we may assume that $\delta \leq \epsilon$. Provided $\delta_0 \leq \delta$, by Proposition \ref{lemma:starting} there exists $B \in [n]^{(t)}$ such that
$$|\mathcal{A} \setminus \mathcal{S}_B| \leq \epsilon {n-t \choose k-t},$$
and therefore
$$|\mathcal{A} \cap \mathcal{S}_B| \geq (1-\delta-\epsilon){n-t \choose k-t} \geq (1-2\epsilon){n-t \choose k-t} \geq {n-t \choose k-t} - \binom{n-t-c}{k-t}.$$
Suppose for a contradiction that
$$|\mathcal{A} \cap \mathcal{S}_B| \leq \binom{n-t}{k-t}-\binom{n-t-d}{k-t}.$$
It follows that $d \geq c$. Provided $n \geq n_0(\eta,t)$ for some sufficiently large $n_0(\eta,t) \in \mathbb{N}$, our assumption that $k \leq (1/(t+1)-\eta)n$ implies that $n > (2+\eta)k+c$, and so by Proposition \ref{lem:main}, we have
$$\left|\mathcal{A}\right|\le\binom{n-t}{k-t}-\binom{n-t-d}{k-t}+\left(2^{t}-1\right)\binom{n-t-d}{k-t-d+1},$$
a contradiction. Hence,
$$|\mathcal{A} \cap \mathcal{S}_B| > \binom{n-t}{k-t}-\binom{n-t-d}{k-t}.$$
Therefore, by Lemma \ref{lemma:union-bound}, we have
$$|\mathcal{A} \setminus \mathcal{S}_B| \leq (2^t-1)\binom{n-t-d}{k-t-d+1},$$
as required.
\end{proof}

\subsection{A $k$-uniform stability result for triangle-intersecting families of graphs}

To obtain a $k$-uniform analogue of Corollary~\ref{Thm:New-Simonovits-Sos}, we just need the following analogue of Proposition \ref{lemma:starting}.
\begin{lemma}
\label{lemma:starting-triangle}
For any $\eta,\epsilon >0$, there exist $\delta = \delta(\epsilon,\eta)>0$ and $n_0 = n_0(\epsilon,\eta) \in \mathbb{N}$ such that the following holds. Let $k,n\in\mathbb{N}$ with $n \geq n_0$ and $k\leq(\tfrac{1}{2}-\eta){n \choose 2}$. Let $\A \subset ([n]^{(2)})^{(k)}$ be a triangle-intersecting family of $k$-edge graphs with vertex-set $[n]$, such that
$$|\mathcal{A}| \geq (1-\delta){{n\choose 2}-3 \choose k-3}.$$
Then there exists a triangle $T$ such that
$$|\mathcal{A} \setminus \mathcal{S}_{T}| \leq \epsilon {{n\choose 2}-3 \choose k-3}.$$
\end{lemma}
\begin{proof}
The proof is almost exactly the same as the $t=3$ case of the proof of Proposition \ref{lemma:starting} (applied with ${n \choose 2}$ in place of $n$), except that we choose $p_0 = 1/2$ instead of $p_0 = 1/(t+1) = 1/4$, and apply Corollary \ref{Thm:New-Simonovits-Sos} instead of Theorem \ref{thm:t-intersecting2-intro}; the details are omitted.
\end{proof}
Applying Lemma \ref{lemma:starting-triangle}, and the $t=3$ case of Proposition \ref{lem:main} (with ${n \choose 2}$ in place of $n$), we get the following.

\begin{theorem}
\label{thm:stability-uniform-triangle}
For any $\eta >0$, there exist $\delta_{0}=\delta_0(\eta)>0$ and $n_{0}=n_{0}(\eta) \in \mathbb{N}$ such that the following holds. Let $k,n \in \mathbb{N}$ with $n \geq n_0$ and $k \leq (\tfrac{1}{2}-\eta){n \choose 2}$, and let $d\in\mathbb{N}$. Let $\mathcal{A}\subset ([n]^{(2)})^{(k)}$ be a triangle-intersecting family of $k$-edge graphs with vertex-set $[n]$, such that
\[
\left|\A\right| > \max\left\{ \binom{{n \choose 2}-3}{k-3}\left(1-\delta_{0}\right),\binom{{n \choose 2}-3}{k-3}-\binom{{n\choose 2}-d-3}{k-3}+7\binom{{n \choose 2}-d-3}{k-d-2}\right\}.
\]
Then there exists a triangle $T$ such that
\[
\left|\mathcal{A} \setminus \mathcal{S}_{T}\right| \leq 7\binom{{n \choose 2}-d-3}{k-d-2}.
\]
\end{theorem}
Note that this result is stronger than the stability theorem for triangle-intersecting families of $k$-edge graphs presented in \cite{EFF12}.

\subsection{A $k$-uniform stability result for the Erd\H{o}s matching conjecture}

In this subsection we prove the following stability result for the Erd\H{o}s matching conjecture; this can be seen as a stability version of Frankl's theorem (Theorem \ref{Thm:Frankl-matching}).
\begin{theorem}\label{Thm:k-uniform-matching}
For any $s \in \mathbb{N}$, $\eta >0$ and $\epsilon >0$, there exists $\delta = \delta(\epsilon,s,\eta)>0$ such that the following holds. Let $n,k \in \mathbb{N}$ with $k \leq (\tfrac{1}{2s+1}-\eta)n$. Suppose $\mathcal{A} \subset [n]^{(k)}$ with $m(\mathcal{A}) \leq s$ and
$$|\mathcal{A}| \geq {n \choose k} - {n-s \choose k} - \delta {n-s \choose k-1}.$$
Then there exists $B \subset [n]^{(s)}$ such that
$$|\mathcal{A} \setminus \OR_{B}| \leq \epsilon {n-s \choose k}.$$
\end{theorem}

\begin{proof}
By the equality case of Theorem \ref{Thm:Frankl-matching}, we have $|\mathcal{A}| = {n \choose k} - {n-s \choose k}$ only if $\mathcal{A} = \OR_{B} \cap [n]^{(k)}$ for some $B \in [n]^{(s)}$, and thus, by making $\delta$ smaller if necessary, we may assume throughout that $n \geq n_0(\epsilon,s,\eta)$ for any $n_0(\epsilon,s,\eta) \in \mathbb{N}$. Choose $\delta = \delta(\epsilon,s,\eta)>0$ such that
$$\delta {n-s \choose k-1} \leq {n-s-c \choose k-1}$$
for all $n \geq n_0$ and all $k \leq n/(2s+1)$, where $c=c(\epsilon,s,\eta) \in \mathbb{N}$ is to be chosen later.

\mn Suppose $\mathcal{A} \subset [n]^{(k)}$ with $m(\mathcal{A}) \leq s$ and
$$|\mathcal{A}| \geq {n \choose k} - {n-s \choose k} - \delta {n-s \choose k-1}.$$
Then
$$|\mathcal{A}| \geq {n \choose k} - {n-s \choose k} - {n-s-c \choose k-1}.$$
Let $\mathcal{F} = \mathcal{A}^{\uparrow}$, and define
$$\mathcal{C} := \{C \subset [n]:\ C \cap [s-1] \neq \emptyset\} \cup \{C \subset [n]:\ C \cap [s] = \{s\},\ C \cap \{s+1,\ldots,s+c\} \neq \emptyset\}.$$
We have
$$|\mathcal{F}^{(k)}| = |\mathcal{A}| \geq  {n \choose k} - {n-s \choose k} - {n-s-c \choose k-1} = |\mathcal{C}^{(k)}|.$$

Since $\mathcal{C}^{(k)}$ is an initial segment of the lexicographic ordering on $[n]^{(k)}$, by the Kruskal-Katona theorem
\begin{equation}\label{eq:dom-3} |\mathcal{F}^{(l)}| \geq |\mathcal{C}^{(l)}|\quad \forall l \geq k.\end{equation}

Define $p_1 := k/n$, $p_0 := 1/(2s+1)$, and $p := \tfrac{1}{2}(p_0+p_1)$. By the Chernoff bound in Proposition \ref{prop:chernoff}, we have
\begin{equation}\label{eq:chernoff-3} \mu_p(\{F \subset [n]:\ |F| < k\}) = \Pr\{\Bin(n,p) < k\} = o_{\eta,s}(1),\end{equation}
and therefore, combining (\ref{eq:dom-3}) and (\ref{eq:chernoff-3}), we obtain
$$\mu_p(\mathcal{F}) \geq \mu_p(\mathcal{C}) - \mu_p(\{F \subset [n]:\ |F| > k\})= 1-(1-p)^s - p(1-p)^{s+c-1} - o_{\eta,s}(1).$$

Since $m(\mathcal{F}) = m(\mathcal{A}) \leq s$, we have $\mu_{p_0}(\mathcal{F}) \leq 1-(1-p_0)^s$. Hence, by Corollary \ref{Cor:Matching}, if
\begin{equation}
\label{eq:cond-0-uniform} (1-p)^{c} + o_{\eta,s}(1) \leq \Theta_s(1) \left( \tfrac{1}{2s+1}-p\right)(1-(1-p)^s),\end{equation}
then there exists $B \in [n]^{(s)}$ such that
$$\mu_p(\mathcal{F} \setminus \OR_{B}) \leq (1-p)^s \epsilon_1,$$
where $\epsilon_1>0$ is the minimal positive solution to
$$(1-p)^{c} + o_{\eta,s}(1) = \tilde{c} \epsilon_1^{\log_{p}(1/(2s+1))\log_{2s/(2s+1)}\left(1-p\right)}- \frac{1-p}{p}\epsilon_1,$$
and $\tilde{c}:=\left(2s\right)^{\log_{2s/(2s+1)}\left(1-p\right)}$. Provided $c$ and $n$ are sufficiently large depending on $\eta$ and $s$, (\ref{eq:cond-0-uniform}) does indeed hold, and crudely, we have
$$\epsilon_1 = O_{s,\eta}(1)(1-p)^{\Theta_{\eta,s}(c)} + o_{\eta,s}(1),$$
and therefore
\begin{equation}\label{eq:upper-bound-or} \mu_p(\mathcal{F} \setminus \OR_{B}) \leq (1-p)^s (O_{s,\eta}(1)(1-p)^{\Theta_{\eta,s}(c)} + o_{\eta,s}(1)).\end{equation}

Without loss of generality, we may assume that $B=[s]$. Choose $d \in \mathbb{N}$ minimal such that
$${n-s-d \choose k-d} \leq \epsilon {n-s \choose k};$$
then $d \leq d_0(s,\epsilon)$ for all $k \leq (1/(2s+1)-\eta)n$. Suppose for a contradiction that
$$|\mathcal{A} \setminus \OR_{B}| > \epsilon {n-s \choose k}.$$
Then
$$|(\F_{[s]}^{\emptyset})^{(k)}|= |\A_{[s]}^{\emptyset}| = |\mathcal{A} \setminus \OR_{B}| > {n-s-d \choose k-d}.$$
Let $\mathcal{D}:= \{D \subset [n] \setminus [s]:\ \{s+1,\ldots,s+d\} \subset D\}$. By the Kruskal-Katona theorem, we have
\begin{equation}\label{eq:dom-4} |(\F_{[s]}^{\emptyset})^{(l)}| \geq  {n-s-d \choose l-d} = |\mathcal{D}^{(l)}| \quad \forall l \geq k.\end{equation}
By the Chernoff bound in Proposition \ref{prop:chernoff}, we have
\begin{equation}\label{eq:chernoff-4} \mu_p(\{F \subset [n]\setminus[s]:\ |F| < k\}) = \Pr\{\Bin(n-s,p) < k\} = o_{\eta,s}(1).\end{equation}
Combining (\ref{eq:dom-4}) and (\ref{eq:chernoff-4}) yields
$$\mu_p(\mathcal{F}_{[s]}^{\emptyset}) \geq \mu_p(\mathcal{D}) - \mu_p(\{F \subset [n]\setminus[s]:\ |F| < k\}) = p^d - o_{\eta,s}(1).$$
Therefore,
$$\mu_p(\mathcal{F}\setminus \OR_{[s]}) \geq (1-p)^s(p^d - o_{\eta,s}(1)).$$
This contradicts (\ref{eq:upper-bound-or}), provided $c$ is sufficiently large depending on $s,\eta$ and $d_0(s,\epsilon)$, and $n$ is sufficiently large depending on $s$ and $\eta$. This completes the proof.
\end{proof}

\begin{remark} In Theorem~\ref{Thm:k-uniform-matching}, the relation between $\delta$ and $\epsilon$ is not specified. Currently, we are able to prove
(by a more complex argument) that the theorem holds with $\epsilon = (c\delta)^{\log_{1-sp_1}p_1}$, where $c=c(s,\eta)>0$, and $p_1:=k/n$. However, we believe that
the right dependence is $\epsilon = (c\delta)^{\log_{1-p_1}p_1}$; this would follow from Conjecture \ref{conj:emc-stab}.
\end{remark}

\section{A comparison with some prior results}
\label{sec:intersecting}

A central feature of several of our results is replacement of the ($t$-)intersection assumption of EKR-type theorems with a weaker assumption (specifically, an upper bound on $\mu_{p_0}(\F)$ for some $p_0$). In this section, we compare our results with some prior results on stability for intersection problems.

\subsection{Intersecting families}
\label{sec:sub:intersecting}

One of the strongest known stability results for the EKR theorem is Frankl's theorem from 1987~\cite{Frankl87}, briefly mentioned in the introduction. To state it in full, we need some more definitions.

If $\F \subset \p([n])$, we define $\deg(\F) : = \max_{j \in [n]} |\{F \in \F:\ j \in F\}|$. For $2 \leq k \leq n-1$ and $3 \leq i \leq k+1$, we define
\[
\G_i := \{A \in [n]^{(k)}: (1 \in A) \wedge (A \cap \{2,3,\ldots,i\} \neq \emptyset)\} \cup \{A \in [n]^{(k)}: (1 \not \in A) \wedge (\{2,3,\ldots,i \} \subset A)\}.
\]
Clearly, each $\G_i$ is an intersecting family.
\begin{theorem}[Frankl, 1987]\label{Thm:Frankl87}
Let $n,k,i \in \mathbb{N}$ with $n>2k$ and $3 \leq i \leq k+1$. Let $\F \subset [n]^{(k)}$ be an intersecting family with $\mathrm{deg}(\F) \leq \mathrm{deg}(\G_i)$. Then $|\F| \leq |\G_i|$.
\end{theorem}

\noindent Theorem~\ref{Thm:Frankl87} implies the following biased-measure version, via the method of `going to infinity and back'. For $3 \leq i \leq n$, we define
\[
\tilde{\G}_i := \{A \subset \p([n]): (1 \in A) \wedge (A \cap \{2,3,\ldots,i\} \neq \emptyset)\} \cup \{A \subset \{0,1\}^n: (1 \not \in A) \wedge (\{2,3,\ldots,i \} \subset A)\}.
\]
Clearly, we have
$$\mu_p(\tilde{\G}_i) = p(1-(1-p)^{i-1})+(1-p)p^{i-1},\quad \mu_p(\tilde{\G}_i \setminus \F_1) = (1-p)p^{i-1}.$$
\begin{corollary}\label{Cor:Frankl0}
Suppose that $0<p<1/2$ and $3 \leq i \leq n$. Let $\F \subset \p([n])$ be an intersecting family with $\mu_p(\F) > \mu_p(\tilde{\G}_i) =p(1-(1-p)^{i-1})+(1-p)p^{i-1}$. Then there exists a dictatorship $\F_j$ such that $\mu_p(\F \cap \F_j) > \mu_p(\tilde{\G}_i \cap \F_1) = p(1-(1-p)^{i-1})$.
\end{corollary}

\noindent An application of Lemma \ref{Lemma:Cross-Intersecting}, together with the observation that $\F_{\{j\}}^{\emptyset}$ and $\F_{\{j\}}^{\{j\}}$ are cross-intersecting for all $j \in [n]$, yields the following.

\begin{corollary}\label{Cor:Frankl1}
Suppose that $0<p<1/2$ and $3 \leq i \leq n$. Let $\F \subset \p([n])$ be an intersecting family with $\mu_p(\F) > \mu_p(\tilde{\G}_i) = p(1-(1-p)^{i-1})+(1-p)p^{i-1}$. Then there exists a dictatorship $\F_j$ such that $\mu_p(\F \setminus \F_j) < \mu_p(\tilde{\G}_i \setminus \F_1) = (1-p)p^{i-1}$.
\end{corollary}

\mn Comparison of our Theorem~\ref{Thm:Biased1-intro} with the corollary of Frankl's result (i.e., Corollary \ref{Cor:Frankl1}), shows that the case $\epsilon=p^{i-1}$ of the Frankl corollary implies the same case of our theorem. On the one hand, the Frankl corollary has an important advantage over our theorem: it applies whenever $\mu_{p}(\F) > \mu_p(\tilde{\G}_3)=3p^2-2p^3$. Theorem~\ref{Thm:Biased1-intro} applies only under the condition (\ref{eq:hypo-intro}), i.e. for $\mu_p(\F) \geq Cp^2$ (for a sufficiently large absolute constant $C$) when $p$ is small, and for $\mu_p(\F) \geq p(1-c(1/2-p))$ (for a sufficiently small absolute constant $c$) when $p$ is large. On the other hand, Theorem~\ref{Thm:Biased1-intro} has two advantages over Frankl's: firstly, we only assume that $\F$ is increasing and that $\mu_{1/2}(\F) \leq 1/2$, which is weaker than the intersection assumption of Frankl. Secondly, for any $\epsilon$ which is not of the form $p^{i-1}$, our result is stronger than Frankl's, provided the condition (\ref{eq:hypo-intro}) holds.

\subsection{$t$-intersecting families, for $t >1$}

Ahlswede and Khachatrian obtained in \cite{AK96} a stability result for the AK theorem which applies to families of size {\em very} close to the maximum. However, the only previously known stability result for Wilson's theorem which applies for families of size within a constant fraction of the maximum, is Friedgut's Theorem~\ref{thm:Fr-uniform}. Our Theorem \ref{thm:stability-uniform-t} implies a strengthening of Theorem \ref{thm:Fr-uniform}, with $\epsilon^{\log_{1-k/n}(k/n)}$ replacing $\epsilon$ in the conclusion of Theorem \ref{thm:Fr-uniform}, i.e.\ sharp $\epsilon$-dependence. It is interesting to note that unlike the proof of Friedgut's theorem, our proof of Theorem \ref{thm:stability-uniform-t} does not rely upon Fourier analysis or spectral techniques.

\subsection{The Erd\H{o}s matching conjecture}
Frankl (unpublished; see \cite{Frankl13}) has proved the following Hilton-Milner type result for the Erd\H{o}s matching conjecture. If $x_0,x_1,\ldots,x_{s-1} \in [n]$ are distinct, and $T_1,\ldots,T_{s} \in [n]^{(k)}$ are pairwise disjoint with $x_i \in T_i$ for all $i \in [s-1]$ and $x_0 \notin T_1 \cup T_2 \cup \ldots \cup T_{s}$, we define the family
\begin{align*} \mathcal{E}(n,k,s) &= \{S \in [n]^{(k)}:\ \exists i \in \{0\}\cup[s-1] \text{ such that } x_i \in S,\ S \cap (T_{i+1} \cup \ldots \cup T_ {s}) \neq \emptyset\}\\
& \cup \{T_1,\ldots,T_s\}.\end{align*}
\begin{theorem}[Frankl]
For any $k,s \in \mathbb{N}$ with $k \geq 4$, there exists $n_1 = n_1(k,s) \in \mathbb{N}$ such that the following holds. Let $n \in \mathbb{N}$ with $n \geq n_1$, and let $\F \subset [n]^{(k)}$ with $m(\F)=s$ and $m(\F_{\{i\}}^{\emptyset}) = s$ for all $i \in [n]$. Then $|\F| \leq |\mathcal{E}(n,k,s)|$, with equality if and only if $\F$ is isomorphic to $\mathcal{E}(n,k,s)$.
\end{theorem}
He has also proved a similar result for $k=3$ (unpublished; see \cite{Frankl13}).

Kostochka and Mubayi \cite{KM16} have recently proved another stability result for the Erd\H{o}s matching conjecture (Theorem 10 in \cite{KM16}), together with several more stability results for Erd\H{o}s-Ko-Rado type problems. Their proofs rely on the Deza-Erd\H{o}s-Frankl delta-system method~\cite{DEF}, and their results therefore only apply for $n \geq n_1(k,s)$, where $n_1(k,s)$ is at least exponential in $k$ (for each $s \in \mathbb{N}$). In this range, their results do not imply ours and are not implied by ours.

To the best of our knowledge, Theorem \ref{Thm:k-uniform-matching} is the first stability result for the Erd\H{o}s matching conjecture which applies when $k = \Theta(n)$ (indeed, whenever $k/n$ is bounded away from $\tfrac{1}{2s+1}$).

\section{Problems for further research}
\label{sec:open-problems}

\mn \textbf{Tighter stability for EKR-type problems.} In Theorem \ref{Thm:Main}, the relation we obtain between $\mu_p(\F)$ and $\mu_p(\F \setminus \s_B)$ is tight. However, in all applications where the
assumptions involve intersection properties (e.g., Theorems~\ref{thm:t-intersecting2-intro},~\ref{Thm:k-uniform-matching} and
Corollaries~\ref{Thm:New-Simonovits-Sos},~\ref{Cor:Matching}) we believe that our results are not tight in some of the
parameters. We conjecture the following strengthening of Theorem \ref{thm:t-intersecting2-intro}.
\begin{conjecture}
Let $t \in \mathbb{N}$, let $0 <p<1/(t+1)$, and let $\mathcal{F} \subset \p([n])$ be a $t$-intersecting family such that
$$\mu_p(\F) \geq (t+2)p^{t+1} - (t+1)p^{t+2}.$$
Let $\epsilon>0$. If
$$\mu_{p}\left(\mathcal{F}\right)\ge p^t\left(1-\left(\frac{\epsilon}{t}\right)^{\log_{p}(1-p)}\right)+ (1-p)p^{t-1}\epsilon,$$
then there exists a $t$-umvirate $\s_B$ such that $\mu_p(\F \setminus \s_B) \leq (1-p)p^{t-1}\epsilon$.
\end{conjecture}
\noindent This would be sharp, as evidenced by the families $\{\tilde{\F}_{t,s}\}_{t,s \in \mathbb{N}}$ defined by
\begin{align*}
\tilde{\mathcal{F}}_{t,s} & :=\left\{ A\subset \mathcal{P}([n])\,:\,\left[t\right]\subset A,\,\left\{t+1,\ldots,t+s\right\}\cap A\ne\emptyset\right\} \\
 & \cup\left\{ A\subset \mathcal{P}([n])\,:\,|\left[t\right]\cap A|=t-1,\,\left\{t+1,\ldots,t+s\right\}\subset A\right\}.
\end{align*}

\mn Likewise, we conjecture the following strengthening of Theorem \ref{thm:stability-uniform-t}.
\begin{conjecture}
Let $t,k,n \in \mathbb{N}$ such that $n \geq (t+1)(k-t+1)$, and let $d\in\mathbb{N}$. If $\mathcal{A}\subset [n]^{(k)}$ is a $t$-intersecting family with
\begin{align*}
\left|\A\right| \geq \max\{&(t+2) {n-t-2 \choose k-t-1} - (t+1) {n-t-2 \choose k-t-2},\\
&\binom{n-t}{k-t}-\binom{n-t-d}{k-t}+t\binom{n-t-d}{k-t-d+1}\},\end{align*}
there exists a $t$-umvirate $\s_B$ such that
\[
\left|\mathcal{A} \setminus \mathcal{S}_{B} \right| \leq t\binom{n-t-d}{k-t-d+1}.
\]
\end{conjecture}
\noindent This would be sharp, as evidenced by the families $\{\tilde{\F}_{t,s} \cap [n]^{(k)}\}_{t,s \in \mathbb{N}}$ (which are the families $\F_{t,s}$ defined in the introduction).

\medskip It would also be of interest to prove tight versions of Corollary \ref{Thm:New-Simonovits-Sos} and of Theorem \ref{thm:stability-uniform-triangle}, regarding triangle-intersecting families of graphs.

It seems that additional tools that exploit the intersection properties more fully, would be needed to prove the above.

\mn \textbf{Stability for families with measure not-so-close to the maximum.} In the case $p_0=1/2$, we conjecture that the condition (\ref{Eq:Condition0}) in Theorem \ref{Thm:Main} could be replaced by the condition $\mu_p(\F) \geq (t+2)p^{t+1}-(t+1)p^{t+2}=\mu_p(\{S\subset [n]:\ |S \cap [t+2]| \geq t+1\})$. 
It would also be of interest to determine, for each $p_0 \in (0,1)$, the sharp analogue of the condition (\ref{Eq:Condition0}) in Theorem \ref{Thm:Main}.

%

\mn \textbf{Stability in cases where the extremal example is not a $t$-umvirate or its dual.} It seems that
the techniques used in this paper are applicable only in cases where the extremal family (corresponding to $\epsilon=0$), or its dual, is a family for which equality holds in the biased edge-isoperimetric inequality on the hypercube (Theorem \ref{thm:skewed-iso}), i.e., the extremal family must be a $t$-umvirate or its dual. It would be interesting to see whether these techniques can be adapted to cases where the extremal family is
more complex, e.g., the full Ahlswede-Khachatrian theorem, where the extremal families are isomorphic to
$\F_{n,k,t,r} = \{S: |S \cap [t+2r]| \geq t+r\}$. We note that in a recent work~\cite{KL16+}, the authors established a stability version of the full AK theorem, for $k/n$ bounded away from zero. However,
the techniques used in~\cite{KL16+} are rather different from those we use here,
and the results obtained there do not imply Theorem \ref{thm:stability-uniform-t}, even in the special case where $k/n$ is bounded away from zero. (They do, on the other hand, imply the less sharp Corollary \ref{corr:rough}.)

\mn \textbf{Sharp stability for the Erd\H{o}s matching conjecture.} The most obvious open question in this area is to resolve the Erd\H{o}s matching conjecture. We also conjecture the following strengthening of Theorem \ref{Thm:k-uniform-matching}.
\begin{conjecture}
\label{conj:emc-stab}
Let $s,k,n \in \mathbb{N}$ with $n \geq (s+1)k$. Let $d \in \mathbb{N}$, and suppose $\mathcal{F} \subset [n]^{(k)}$ is such that $m(\F) = s$ and
$$|\F| \geq \max\left\{ {k(s+1)-1 \choose k}+1,{n \choose k} - {n-s \choose k} - {n-s-d \choose k-1} + {n-s-d \choose k-d}\right\}.$$
Then there exists $B \in [n]^{(s)}$ such that
$$|\A \setminus \OR_{B}| \leq {n-s-d \choose k-d}.$$
\end{conjecture}
\noindent This would be sharp, as evidenced by the family
\begin{align*} \mathcal{H} & = \{A \in [n]^{(k)}:\ A \cap [s-1] \neq \emptyset\}\\
& \cup \{A \in [n]^{(k)}:\ A \cap [s] = \{s\},\ A \cap \{s+1,s+2,\ldots,s+d\} \neq \emptyset\}\\
& \cup \{A \in [n]^{(k)}:\ A \cap [s] = \emptyset,\ \{s+1,\ldots,s+d\} \subset A\}.\end{align*}

\subsection*{Acknowledgements}

We are grateful to Gil Kalai for encouraging us to work on this project, and for his motivating questions and suggestions. We also thank the two anonymous referees for their careful reading of the paper and their helpful suggestions and comments.

\begin{thebibliography}{99}

\bibitem{AK77} R. Ahlswede and G. O. Katona, Contributions to the geometry of Hamming spaces, {\it Disc.
Math.}, \textbf{17} (1977), pp. 1--22.

\bibitem{AK96} R. Ahlswede and L. H. Khachatrian, The complete nontrivial-intersection theorem for systems
of finite sets, {\it J. Combin. Theory, Series A} \textbf{76} (1996), pp. 121--138.

\bibitem{AK} R. Ahlswede and L. H. Khachatrian, The complete intersection theorem for systems of finite
sets, {\it Eur. J. Combin.} \textbf{18} (1997), pp. 125--136.

\bibitem{AS} N. Alon and J. Spencer, The probabilistic method, 4th Edition, John Wiley $\&$ Sons, 2016.

\bibitem{AK06} R. P. Anstee and P. Keevash, Pairwise intersections and forbidden configurations,
{\it Eur. J. Combin.} \textbf{27} (2006), pp. 1235--1248.

\bibitem{BBN16+} J. Balogh, B. Bollob\'as, and B. Narayanan, Transference for the Erd\H{o}s-Ko-Rado theorem,
{\it Forum of Mathematics -- Sigma}, \textbf{3} (2015).

\bibitem{BM08} J. Balogh and D. Mubayi, A new short proof of a theorem of Ahlswede and
Khachatrian, {\it J. Combin. Theory, Series A} \textbf{115} (2008), pp. 326--330.

\bibitem{bernstein} A. J. Bernstein, Maximally connected arrays on the $n$-cube, {\em SIAM J. Appl. Math.}
\textbf{15} (1967), pp.1485--1489.

\bibitem{Bollobas} B. Bollob\'{a}s, {\em Combinatorics: set systems, hypergraphs, families of vectors and combinatorial probability}, CUP, Cambridge, 1986.

\bibitem{BDE76}  B. Bollob\'{a}s, D. E. Daykin, and P. Erd\H{o}s, Sets of independent edges of a hypergraph,
{\it Quart. J. Math. Oxford Ser. 2} \textbf{27} (1976), pp. 25--32.

\bibitem{BNR16} B. Bollob\'as, B. Narayanan, and A. Raigorodskii, On the stability of the Erd\H{o}s-Ko-Rado
theorem, {\it J. Combin. Theory, Series A,} \textbf{137} (2016), pp. 64--78.

\bibitem{DK16+} P. Devlin and J. Kahn, On `stability' in the Erd\H{o}s-Ko-Rado theorem, {\it SIAM J. Discrete Math.}, 
\textbf{30(2)} (2016), pp.~1283-–1289.

\bibitem{DEF} M. Deza, P. Erd\H{o}s, and P. Frankl, Intersection properties of systems of finite sets, {\it Proc. London 
Math. Soc.}, \textbf{36(2)} (1978), pp.~369--384. 

\bibitem{DF83} M. Deza and P. Frankl, Erd\H{o}s-Ko-Rado theorem -- 22 years later, {\it SIAM
J. Alg. Disc. Meth.} \textbf{4(4)} (1983), pp. 419--431.

\bibitem{DF06} I. Dinur and E. Friedgut, Proof of an intersection theorem via graph homomorphisms,
{\it Electron. J. Combin.} \textbf{13(1)} (2006), N6.

\bibitem{DF09} I. Dinur and E. Friedgut, Intersecting families are essentially contained in juntas,
{\it Combin. Probab. Comput.} \textbf{18(1-2)} (2009), pp. 107--122.

\bibitem{Dinur-Safra} I. Dinur and S. Safra, On the hardness of approximating minimum vertex-cover,
{\it Ann. of Math.} \textbf{162} (2005), pp. 439--485.


\bibitem{Ellis11} D. Ellis, Almost isoperimetric subsets of the discrete cube, {\it Combin., Probab.
Comput.} \textbf{20(3)} (2011), pp. 363--380.

\bibitem{EFF12} D. Ellis, Y. Filmus, and E. Friedgut, Triangle-intersecting families of graphs,
{\it J. Eur. Math. Soc.} \textbf{14} (2012), pp. 841--885.

\bibitem{EKL16+} D. Ellis, N. Keller and N. Lifshitz, On a biased edge-isoperimetric inequality for the discrete cube, preprint. arXiv:1702.01675.

\bibitem{KL16+} D. Ellis, N. Keller and N. Lifshitz, A tight stability version of the Complete Intersection Theorem, and the forbidden intersection problem of Erd\H{o}s and S\'os, preprint. arXiv:1604:06135.

\bibitem{EKR} P. Erd\H{o}s, C. Ko, and R. Rado, Intersection theorems for systems of finite sets, {\it
Quart. J. Math. Oxford Ser. 2}, \textbf{12} (1961), pp. 313--320.

\bibitem{Erdos65}  P. Erd\H{o}s, A problem on independent $r$-tuples, {\it Ann. Univ. Sci. Budapest},
\textbf{8} (1965), pp. 93--95.

\bibitem{F07}  Y. Filmus, Proof of the $\mu_p$ version of the Erd\H{o}s-Ko-Rado theorem using Katona's
method, unpublished. Available at \texttt{http://www.cs.toronto.edu/~yuvalf/}.

\bibitem{FFFLO86} P. C. Fishburn, P. Frankl, D. Freed, J. C. Lagarias, and A. M. Odlyzko, Probabilities
for intersecting systems and random subsets of finite sets, {\it SIAM. J. Alg. Disc. Meth.},
\textbf{7(1)} (1986), pp. 73--79.

\bibitem{Frankl87} P. Frankl, Erd\H{o}s-Ko-Rado theorem with conditions on the maximal degree,
{\it J. Combin. Theory, Series A} \textbf{46} (1987), pp. 252--263.

\bibitem{Frankl13} P. Frankl, Improved bounds for Erd\H{o}s' Matching Conjecture, {\it J. Combin.
Theory, Series A} \textbf{120} (2013) pp. 1068--1072.

\bibitem{FLST14} P. Frankl, S. J. Lee, M. Siggers, and N. Tokushige, An Erd\H{o}s-Ko-Rado theorem for
cross $t$-intersecting families, {\it J. Combin. Theory, Series A}, \textbf{128} (2014), pp. 207--249.

\bibitem{FT03} P. Frankl and N. Tokushige, Weighted multiply intersecting families, {\it Studia Sci. Math.
Hungarica} \textbf{40} (2003), pp. 287--291.

\bibitem{Friedgut-ksat} E. Friedgut (with an appendix by Jean Bourgain), Sharp thresholds of graph properties, and the $k$-sat problem, {\em J. Amer. Math. Soc.} 12 (1999), 1017--1054.

\bibitem{FKN} E. Friedgut, G. Kalai, and A. Naor, {\it Boolean
functions whose Fourier transform is concentrated on the first two
levels}, Adv. Appl. Math. \textbf{29(3)} (2002), pp. 427--437.

\bibitem{Friedgut08} E. Friedgut, On the measure of intersecting families, uniqueness and
stability, {\it Combinatorica} \textbf{28} (2008), pp. 503--528.

\bibitem{Grimmett} G. R. G. Grimmett, Percolation (2nd edition), Springer-Verlag, 1999.

\bibitem{Harper} L. H. Harper, Optimal assignment of numbers to vertices, {\it SIAM J. Appl. Math.},
\textbf{12} (1964), pp. 131--135.

\bibitem{hart} S. Hart, A note on the edges of the $n$-cube, {\it Disc. Math.} \textbf{14} (1976),
pp. 157--163.

\bibitem{HM67} A. J. W. Hilton and E. C. Milner, Some intersection theorems for systems of finite sets,
{\it Quart. J. Math. Oxford Ser. 2} \textbf{18} (1967), pp. 369--384.

\bibitem{HLS12} H. Huang, P. Loh, and B. Sudakov, The size of a hypergraph and its matching number,
{\it Combin. Probab. Comput.} \textbf{21} (2012), pp. 442--450.

\bibitem{Kahn-Kalai} J. Kahn and G. Kalai, Thresholds and expectation thresholds, {\it Combin.,
Probab. Comput.} \textbf{16(3)} (2007), pp. 495--502.


\bibitem{Kamat11} V. Kamat, Stability analysis for $k$-wise intersecting families, {\it Electron. J. Combin.} \textbf{18}
(2011), P115.

\bibitem{Katona64} G. O. H. Katona, Intersection theorems for systems of finite sets, {\it Acta Math. Acad. Sci. Hungar.}
\textbf{15} (1964), pp. 329--337.

\bibitem{Katona66} G. O. H. Katona, A theorem of finite sets, in: `Theory of Graphs', Proc. Colloq. Tihany, Hungary, 1966, pp. 187--207.

\bibitem{Keevash08} P. Keevash, Shadows and intersections: Stability and new proofs, {\it Adv. Math.}
\textbf{218} (2008), pp. 1685--1703.

\bibitem{KM10} P. Keevash and D. Mubayi, Set systems without a simplex or a cluster, {\it Combinatorica},
\textbf{30(2)} (2010), pp. 175--200.

\bibitem{KMW07} P. Keevash, D. Mubayi and R. M. Wilson, Set systems with no singleton intersection, {\it SIAM J. Disc. Math.} \textbf{20} (2007), pp. 1031--1041.

\bibitem{KS03} G. Kindler and S.Safra, Noise-resistant Boolean functions are juntas, manuscript, 2003.

\bibitem{KM16} A. Kostochka and D. Mubayi, The structure of large intersecting families, {\it Proc. Amer. Math. Soc.}
\textbf{145(6)} (2017), pp.~2311--2321.

\bibitem{Kruskal63} J. B. Kruskal, The number of simplices in a complex, in: Mathematical Optimization
Techniques, Univ. California Press, Berkeley, 1963, pp. 251--278.

\bibitem{lindsay} J. H. Lindsey II, Assignment of numbers to vertices, {\it Amer. Math. Monthly}
\textbf{71} (1964), pp. 508--516.


\bibitem{Mubayi07} D. Mubayi, Structure and stability of triangle-free set systems, {\it Trans. Amer. Math. Soc.}, \textbf{359} (2007), pp. 275--291.

\bibitem{MV15} D. Mubayi and J. Verstra\"{e}te, A survey on Tur\'{a}n problems for expansions,
Recent Trends in Combinatorics, IMA Volumes in Mathematics and its Applications 159. Springer,
New York, 2015.

\bibitem{Nayar13} P. Nayar, FKN theorem on the biased cube, {\it Colloq. Math.}, \textbf{137} (2014), pp.~253--261.

\bibitem{Russo} L. Russo, An approximate zero-one law,
{\it Z. Wahrsch. Verw. Gebiete} \textbf{61} (1982), pp. 129--139.

\bibitem{Simonovits68} M. Simonovits, A method for solving extremal problems in graph theory, stability
problems, in `Theory of Graphs', Proc. Colloq. Tihany, Hungary, 1966, pp. 279--319.

\bibitem{Talagrand94} M. Talagrand, On Russo's Approximate 0-1 Law, {\it Ann. Probab.} \textbf{22} (1994),
pp. 1576--1587.

\bibitem{Tokushige10} N. Tokushige, A multiply intersecting Erd\H{o}s-Ko-Rado theorem, the principal case,
{\it Disc. Math.} \textbf{310} (2010), pp. 453--460.

\bibitem{Tokushige11} N. Tokushige, A product version of the Erd\H{o}s-Ko-Rado theorem, {\it J. Combin.
Theory, Series A,} \textbf{118(5)} (2011), pp. 1575--1587.

\bibitem{Tokushige13} N. Tokushige, Cross $t$-intersecting integer sequences from weighted Erd\H{o}s--Ko--Rado, {\it Combin. Probab. Comput.} \textbf{22} (2013), pp. 622--637.

\bibitem{Wilson} R.M. Wilson, The exact bound in the Erd\H{o}s-Ko-Rado theorem, {\it Combinatorica} \textbf{4}
(1984), pp. 247--257.

\end{thebibliography}
\end{document}